\documentclass{amsart}
\usepackage{amsmath,amscd,latexsym, amssymb}
\usepackage{graphicx}
\newtheorem{thm}{Theorem}[section]

\newtheorem{lem}[thm]{Lemma}
\newtheorem{cor}[thm]{Corollary}
\newtheorem{exa}[thm]{Example}
\newtheorem{rmk}[thm]{Remark}
\newtheorem{dfn}[thm]{Definition}

\begin{document}
\title{On $n$-semiprimary ideals and $n$-pseudo valuation domains}
\author{David F. Anderson and  Ayman Badawi}
\address{Department of Mathematics, The University of Tennessee,
Knoxville, TN 37996-1320, U. S. A.}\email{danders5@utk.edu}
 \address{Department of Mathematics  $\&$ Statistics, The American University of
Sharjah, P.O.  Box 26666, Sharjah, United Arab Emirates}\email{abadawi@aus.edu}

\date{\today}

\keywords{prime ideal, 2-absorbing ideal, n-absorbing ideal, primary ideal, semiprimary ideal, radical ideal, powerful ideal, valuation domain, almost valuation domain, pseudo-valuation domain, almost pseudo-valuation domain, pseudo-almost valuation domain.}
\subjclass[2000]{13A15}

\begin{abstract}
Let $R$ be a commutative ring with $1\not = 0$ and $n$  a positive integer. A proper ideal $I$  of $R$ is an {\it $n$-semiprimary ideal} of $R$ if whenever $x^ny^n \in I$ for $x, y \in R$, then $x^n \in I$ or $y^n \in I$.  Let $R$ be an integral domain with quotient field $K$. A proper ideal $I$ of $R$  is an {\it $n$-powerful ideal} of $R$ if whenever $x^ny^n \in I$ for $x, y \in K$, then $x^n \in R$ or $y^n \in R$; and $I$ is an {\it $n$-powerful semiprimary ideal} of $R$ if whenever $x^ny^n \in I$ for  $x, y \in K$, then $x^n \in I$ or $y^n \in I$. If every prime ideal of $R$ is an $n$-powerful semiprimary ideal of $R$, then $R$ is an {\it  $n$-pseudo-valuation domain {\rm (}$n$-PVD{\rm )}}. In this paper, we study the above concepts and relate them to several generalizations of pseudo-valuation domains.
\end{abstract}

\maketitle{}

\section{introduction}

Let $R$ be a commutative ring with $1 \neq 0$ and $n$ a positive integer. Recall that an ideal $I$ of $R$ is a {\it semiprimary ideal} of $R$ if $\sqrt{I}$ is a prime ideal of $R$. In this paper, we introduce and study  $n$-semiprimary ideals (resp., $n$-powerful semiprimary ideals in integral domains), where a proper ideal $I$ of $R$ is {\it $n$-semiprimary} (resp.,
 {\it $n$-powerful semiprimary}) if whenever $x^ny^n \in I$ for $x, y \in R$ (resp., $x,y \in K$, the quotient field of $R$), then $x^n \in I$ or $y^n \in I$. These concepts generalize prime ideals and are generalized by semiprimary ideals. We also investigate several other ``$n$'' generalizations obtained by replacing $x$ with $x^n$ in the definition.

In Section $2$, we give some basic properties of $n$-semiprimary ideals. For example, we show that an $n$-semiprimary ideal is semiprimary, and the converse holds when $R$ is Noetherian. We also show that an $n$-semiprimary ideal is $m$-semiprimary for every integer $m \geq n$. In Section $3$, we characterize $n$-semiprimary ideals in several classes of commutative rings. In particular, we investigate $n$-semiprimary ideals in zero-dimensional commutative rings, Dedekind domains, valuation domains, and idealizations. In  Section $4$, we study $n$-powerful semiprimary ideals in integral domains and introduce $n$-pseudo-valuation domains ($n$-PVDs), a generalization of pseudo-valuation domains (PVDs). We also study $n$-valuation domains ($n$-VDs).  In the final section, Section $5$, we introduce pseudo $n$-valuation domains (P$n$VDs), another generalization of PVDs. Many examples are given throughout the paper to illustrate the theory. 

Throughout, $R$ will be a commutative ring with $1 \neq 0$, $\sqrt{I} = \{ x \in R \mid x^n \in I$ for some $n \in \mathbb{N} \}$ for $I$ an ideal of $R$, ideal of nilpotent elements $nil(R) = \sqrt{\{0\}}$, group of units $U(R)$, (Krull) dimension $dim(R)$, and characteristic $char(R)$. An overring of an integral domain $R$ with quotient field $K$ is a subring of $K$ containing $R$, and we denote the integral closure of $R$ (in $K$) by $\overline{R}$. In particular, if $I$ is an ideal of  $R$,  then $(I:I) = \{ x \in K \mid xI \subseteq I \}$ is an overring of $R$. Other definitions will be given throughout the paper as needed. As usual, $\mathbb{N}$, $\mathbb{Z}$, $\mathbb{Z}_n$, $\mathbb{F}_{p^n}$, $\mathbb{Q}$, $\mathbb{R}$, and $\mathbb{C}$ will denote the set of positive integers, the rings of integers and  integers mod $n$, the finite field with $p^n$ elements, and the fields of rational numbers, real numbers, and complex numbers, repectively. For any undefined terminology, see \cite{G}, \cite{H}, \cite{K}, or \cite {LM}.

\section{Basic properties of $n$-semiprimary ideals} \label{sec2}

In this section, we give some basic properties of $n$-semiprimary ideals. We begin with the definition.

\begin{dfn}
{\rm Let $I$ be a proper ideal of a commutative ring $R$ and $n$ a positive integer. Then $I$ is an {\it $n$-semiprimary ideal} of $R$ if whenever $x^ny^n \in I$ for $x, y \in R$, then $x^n \in I$ or $y^n \in I$.}
\end{dfn}

Note that a $1$-semiprimary ideal is just a prime ideal.  For convenience, call a commutative ring $R$ an {\it $n$-ring} if $x^ny^n = 0$ for $x, y \in R$ implies $x^n = 0$ or $y^n = 0$. Then a $1$-ring is just an integral domain, $R$ is an $n$-ring if and only if $\{0\}$ is an $n$-semiprimary ideal of $R$, and $R/I$ is an $n$-ring  if and only $I$ is an $n$-semiprimary ideal of $R$. We start with some elementary results that follow directly from the definitions. 

\begin{thm} \label{T0}
Let $I$ be a proper ideal of a commutative ring $R$. 

{\rm (a)}  Let $I$ be an $n$-semiprimary ideal of $R$. Then $I$ is an $mn$-semiprimary ideal of $R$ for every positive integer $m$. {\rm (}See Theorem~\ref{T3.5} for a stronger result.{\rm )}

{\rm (b)} Let $J \subseteq I$ be proper ideals of $R$. Then $I$ is an $n$-semiprimary ideal of $R$ if and only if $I/J$ is an $n$-semiprimary ideal of $R/J$. 

{\rm (c)} Let $I$ be an $n$-semiprimary ideal of $R$ and $S$ a multiplicatively closed subset of $R$ with $I \cap S = \emptyset$. Then $I_S$ is an $n$-semiprimary ideal of $R_S$.
\end{thm}

 We next show that an $n$-semiprimary ideal is indeed semiprimary.

\begin{thm} \label{T1}
Let $I$ be an $n$-semiprimary ideal of a commutative ring $R$. Then $\sqrt{I}$ is a prime ideal of $R$ and $x^n \in I$ for every $x \in \sqrt{I}$. In particular, $I$ is a semiprimary ideal of $R$, and $x \in \sqrt{I}$ if and only if $x^n \in I$.
\end{thm}

\begin{proof}
Let $xy \in \sqrt{I}$ for $x, y \in R$. Then there is a positive integer $k$  such that $(x^k)^n(y^k)^n = (xy)^{kn} \in I$.  Thus $x^{kn} = (x^k)^n  \in I$ or $y^{kn} = (y^k)^n  \in I$ since $I$ is an $n$-semiprimary ideal of  $R$. Hence $x \in \sqrt{I}$ or $y \in \sqrt{I}$; so $\sqrt{I}$ is a prime ideal of $R$. Let $x \in \sqrt{I}$ and $m$ be the least positive integer such that $x^{mn} \in I$. Then $x^n(x^{m-1})^n = x^nx^{(m-1)n} = x^{mn}  \in I$, and thus $x^n \in I$ or $x^{(m-1)n} \in I$ since $I$ is an $n$-semiprimary ideal of  $R$. Hence $m = 1$; so $x^n \in I$. The ``in particular'' statement is clear.
\end{proof}

The following is an example of a semiprimary ideal of a commutative ring $R$ that is not an $n$-semiprimary ideal for any positive integer $n$. Note that $R$ is {\it not} Noetherian. In fact, Corollary~\ref{C0} shows that semiprimary ideals in a commutative Noetherian ring are $n$-semiprimary for all large $n$.

\begin{exa} \label{E0}
{\rm Let $R = \mathbb{Z}_2[\{  X_n \}_{n = 1}^{\infty}]$ and $I = (\{  X_n^n \}_{n = 1}^{\infty})$. Then $\sqrt{I} =   (\{  X_n \}_{n = 1}^{\infty})$ is a prime ideal of $R$; so $I$ is a semiprimary ideal of $R$. However, $I$ is not an $n$-semiprimary ideal of $R$ for any positive integer $n$ since $X_{2n}^nX_{2n}^n =X_{2n}^{2n} \in I$, but  $X_{2n}^n \not \in I$.}
\end{exa}

The next theorem gives a sufficient condition for a semiprimary ideal to be an $n$-semiprimary ideal. As a consequence,   $n$-absorbing semiprimary ideals are $n$-semiprimary and semiprimary ideals in commutative Noetherian rings are $n$-semiprimary for all large $n$.

\begin{thm} \label{T2}
Let $I$ be a proper ideal of a commutative ring $R$ such that $P = \sqrt{I}$ is a prime ideal of $R$ and $P^n \subseteq I$ for a positive integer $n$. Then $I$ is an $m$-semiprimary ideal of $R$ for every integer $m \geq n$. In particular, $Q^n$ is an $m$-semiprimary ideal of $R$ for every prime ideal $Q$ of $R$ and integer $m \geq n$.
\end{thm}

\begin{proof}
Let $x^ny^n \in I \subseteq P$ for  $x, y \in R$. Then $x \in P$ or $y \in P$. Thus $x^n \in P^n \subseteq I$ or $y^n \in P^n \subseteq I$, and hence $I$ is an $n$-semiprimary ideal of  $R$.  Moreover, $P^m \subseteq P^n \subseteq I$ for every integer $m \geq n$; so $I$ is also  an $m$-semiprimary ideal of  $R$ for every integer $m \geq n$. The ``in particular'' statement is clear.
\end{proof}

\begin{cor} \label{C0}
Let $I$ be a semiprimary ideal of a commutative Noetherian ring $R$. Then there is a positive integer $n$ such that $I$ is an $m$-semiprimary ideal of $R$ for every integer $m \geq n$.
\end{cor}

\begin{proof}
Since $I$ ​is a semiprimary ideal of $R$, $P = \sqrt{I}$ is a prime ideal of $R$, and  $P^n \subseteq I$ 
for some positive integer $n$ since $P$ is finitely generated. Thus $I$ is an $m$-semiprimary ideal of $R$ for every integer  $m \geq n$ by Theorem \ref{T2}.
\end{proof}

Recall \cite{AB1} that a proper ideal $I$ of a commutative ring $R$ is an {\it $n$-absorbing ideal} of $R$ if whenever $x_1 \cdots x_{n+1} \in I$ for $x_1, \ldots, x_{n+1} \in R$, then the product of $n$ of the $x_i$'$s$ is in $I$. Both $n$-semiprimary and $n$-absorbing ideals generalize prime ideals, but in rather different ways. An $n$-semiprimary ideal need not be an $n$-absorbing ideal (see Example~\ref{E1}); and an $n$-absorbing ideal need not be $n$-semiprimary since, for example, $(6)$ is a $2$-absorbing ideal of $\mathbb{Z}$, but not a $2$-semiprimary ideal since $\sqrt{(6)} = (6)$ is not a prime ideal of $\mathbb{Z}$. However, we next show that if $\sqrt{I}$ is a prime ideal, then an $n$-absorbing ideal $I$ is $n$-semiprimary.

\begin{cor} \label{C1}
Let $I$ be an $n$-absorbing ideal of a commutative ring $R$. If $\sqrt{I}$ is a prime ideal of $R$, then $I$ is an $m$-semiprimary ideal of $R$ for every integer $m \geq n$. In particular, an $n$-absorbing ideal is $n$-semiprimary if and only if it is semiprimary.
\end{cor}

\begin{proof}
Let $P = \sqrt{I}$ be a prime ideal of $R$.  Then $P^n = (\sqrt{I})^n \subseteq I$ since $I$ is an $n$-absorbing ideal of $R$ (\cite{CW},  \cite{D}). Thus $I$ is an $m$-semiprimary ideal of  $R$ for every integer $m \geq n$ by Theorem \ref{T2}. The ``in particular'' statement now  follows from Theorem~\ref{T1}. 
\end{proof}

\begin{cor} \label{C2}
Let $P_1 \subseteq \cdots  \subseteq P_k$ be prime ideals of a commutative ring $R$ and $n_1, \ldots, n_k$ positive integers. Then $I = P_1^{n_1} \cdots P_k^{n_k}$ is an $m$-semiprimary ideal of $R$ for every integer $m \geq n_1 + \cdots+  n_k$.
\end{cor}

\begin{proof}
Note that $\sqrt{I} = P_1$ is a prime ideal of $R$ and $P^n_1 \subseteq  P_1^{n_1} \cdots P_k^{n_k} = I$, where $n = n_1 + \cdots + n_k$. Thus $I$ is an $m$-semiprimary ideal of  $R$  for every integer $m \geq n$ by Theorem \ref{T2}.
\end{proof}

The converse of Theorem \ref{T2} need not be true, i.e., if $I$ is an $n$-semiprimary ideal of  $R$ for some integer $n \geq 2$, then $(\sqrt{I})^n$ need not be a subset of $I$. Let $p \geq 2$ be a prime integer. In the following example, we show that there is a proper ideal $I$ of a commutative ring $R$ such that $I$ is a $p$-semiprimary ideal of  $R$, but $(\sqrt{I})^p \nsubseteq I$, and thus $I$ is not a $p$-absorbing ideal of $R$ (\cite{CW}, \cite{D}).

\begin{exa} \label{E1}
{\rm Let $p \geq 2$ be a prime integer, $R = \mathbb{Z}_p[X, Y]$, and $I = (X^p, Y^p)$. Then $I$ is a proper ideal of $R$ with prime ideal $P = \sqrt{I} = (X, Y)$ and $P^p \nsubseteq I$ since $YX^{p-1} \notin I$. Thus $I$ is not a $p$-absorbing ideal of $R$ (\cite{CW}, \cite{D}). Let $f^pg^p \in I \subseteq (X, Y)$ for  $f, g \in R$. Then $f \in (X, Y)$ or $g \in (X, Y)$; so $f^p \in I$ or $g^p \in I$, and hence $I$ is a $p$-semiprimary  ideal of $R$.}
\end{exa}

Recall \cite{CH} that a proper ideal $I$ of a commutative ring $R$ is a {\it uniformly primary ideal} of $R$ if there is a positive integer $n$ such that whenever $xy \in I$ for $x, y \in R$, then $x \in I$ or $y^n \in I$. If $I$ is a uniformly primary ideal of $R$ for a positive integer $n$, then we say that $I$ is an {\it n-primary ideal} of $R$. By the following theorem, an $n$-primary ideal is also $n$-semiprimary.

\begin{thm} \label{T3}
Let $I$ be an $n$-primary ideal of a commutative ring $R$. Then $I$ is an $n$-semiprimary ideal of  $R$.
\end{thm}

\begin{proof}
Let $x^ny^n \in I$ for $x, y \in R$ with $x^n \notin I$, and let $m$ be the least positive integer such that $x^ny^m \in I$. Then $(x^ny^{m-1})y = x^ny^m \in I$. Since $x^ny^{m-1} \notin I$ and $I$ is an $n$-primary ideal of $R$, we have $y^n \in I$. Thus $I$ is an $n$-semiprimary  ideal of $R$.
\end{proof}

In the following example, we show that there is a commutative ring $R$ with ideals $\{ I_n \}_{n=2}^{\infty}$ such that every $I_n$ is an $n$-semiprimary ideal of  $R$ with $(\sqrt{I_n})^n \subseteq I_n$, but $I_n$ is not a primary ideal of $R$. In particular,  $I_n$ is not an $m$-primary ideal of $R$ for any positive integer $m$.

\begin{exa}\label{E2}
{\rm Let  $R = \mathbb{Z}_2[X, Y]$. For every integer $n \geq 2$, $I_n = (XY, Y^n)$ is an ideal of $R$ with prime ideal $P = \sqrt{I_n} = (Y)$. Thus $I_n$ is an $n$-semiprimary ideal of  $R$ by Theorem \ref{T2} since $P^n \subseteq I_n$. However, $YX \in I_n$, $Y\notin I_n$, and $X^m \not \in I_n$ for every positive integer $m$; so $I_n$ is not a primary ideal of $R$, and hence $I_n$ is not an $m$-primary ideal of $R$ for any positive integer $m$.}
\end{exa}

The next definition generalizes the ``$n$-semiprimary'' concept from elements to ideals.

\begin{dfn}
{\rm Let $I$ be a proper ideal of a commutative ring $R$ and $n$ a positive integer. Then $I$ is a {\it strongly $n$-semiprimary ideal} of $R$ if whenever $J^nK^n \subseteq I$ for proper ideals $J$ and $K$ of $R$, then $J^n \subseteq I$ or $K^n \subseteq I$.}
\end{dfn}

A strongly $1$-semiprimary ideal is just a prime ideal,  a strongly $n$-semiprimary ideal is an $n$-semiprimary ideal, and a strongly $n$-semiprimary ideal is also strongly $mn$-semiprimary for every positive integer $m$. However, the following example shows that an $n$-semiprimary ideal need not be strongly $n$-semiprimary.

\begin{exa}\label{E3}
{\rm Let $R = \mathbb{Z}_2[X, Y]$ and $I = (X^2, Y^2)$. By Example~\ref{E1}, $I$ is a $2$-semiprimary ideal of $R$ with prime ideal $P = \sqrt{I} = (X, Y)$. Clearly, $P^2P^2 = P^4 \subseteq I$, but $P^2 \nsubseteq I$. Thus $I$ is not a strongly $2$-semiprimary ideal of $R$. Note that $I$ is an $n$-semiprimary ideal of $R$ for every integer $n \geq 3$ by Theorem~\ref{T2} since $P^3 \subseteq I$, and hence $I$ is an $n$-semiprimary ideal of $R$ for every integer $n \geq 2$.}
\end{exa}

We have already observed in Theorem~\ref{T0} that an $n$-semiprimary ideal is also $mn$-semiprimary for every positive integer $m$. We next give a much stronger result.

\begin{thm} \label{T3.5}
Let $I$ be an $n$-semiprimary ideal of a commutative ring $R$.

{\rm (a)} If $x^my^k \in I$ for $x, y \in R$ and positive integers $m$ and $k$, then $x^n \in I$ or $y^n \in I$. In particular, if $x^m \in I$ for $x \in R$ and $m$ a postive integer, then $x^n \in I$.

{\rm (b)} $I$ is an $m$-semiprimary ideal of $R$ for every positive integer $m \geq n$.
\end{thm}

\begin{proof}
(a)  Let $x^my^k \in I$ for $x, y \in R$; we may assume that $m \geq k$. Then $(xy)^m = x^my^m = (x^my^k)y^{m-k} \in I$ . Thus $xy \in\sqrt{ I}$; so $x^ny^n = (xy)^n \in I$ by Theorem~\ref{T1}. Hence $x^n \in I$ or $y^n \in I$ since $I$ is an $n$-semiprimary ideal of $R$. The ``in particular'' statement is clear.

(b) Let $x^my^m \in I$ for $x, y \in R$ with $m \geq n$. Then $x^n \in I$ or $y^n \in I$ by part (a). Thus  $x^m = x^{m-n}x^n \in I$ or $y^m = y^{m-n}y^n \in I$ since $m \geq n$; so $I$ is an $m$-semiprimary ideal of $R$.
\end{proof}

An ideal may be $n$-semiprimary for many different values of $n$. We now make that statement more precise. For a proper ideal $I$ of a commutative ring $R$, let $W_R(I) = \{ n \in \mathbb{N} \mid I$ is an $n$-semiprimary ideal of $R \}$ and $\delta_R(I) = min W_R(I)$ (let $\delta_R(I) = \infty$ if  $W_R(I) =\emptyset$). Then  $W_R(I) =  [\delta_R(I), \infty) \cap \mathbb{N}$ by Theorem~\ref{T3.5}(b).

\section{$n$-semiprimary ideals in some classes of rings}

In this section, we study $n$-semiprimary ideals in several important classes of commutative rings. We have already observed in Corollary~\ref{C0} that for commutative Noetherian rings, a semiprimary ideal is $n$-semiprimary for all large  $n$. ‭The first two results concern the case when $dim(R) = 0$.

\begin{thm}\label{S3T2}
Let $I \supseteq nil(R)$ be an ideal of a commutative ring $R$ with $dim(R) = 0$. Then $I$ is an $n$-semiprimary ideal of $R$ if and only if $I$ is a prime ideal of $R$ {\rm (}i.e., $I$ is a $1$-semiprimary ideal of $R${\rm )}.
\end{thm}

\begin{proof}
A prime ideal is certainly $n$-semiprimary for every positive integer $n$. Conversely, we show that an $n$-semiprimary ideal $I$ of $R$ is a prime ideal of $R$.  Let $xy \in I$ for  $x, y \in R$; so $x^ny^n \in I$. Then $x^n \in I$ or $y^n \in I$; say $x^n \in I$. Since $dim(R) = 0$, we have $x = eu + w$ for an idempotent $e \in R$, $u \in U(R)$, and $w \in nil(R)$ \cite[Corollary 1]{B}. Thus $x^n = (eu + w)^n = eu^n + a_1eu^{n-1}w + a_2eu^{n-2}w^2 + \cdots + a_{n-1}euw^{n-1} + w^n = e(u^n + a_1u^{n-1}w + a_2u^{n-2}w^2 + \cdots + a_{n-1}uw^{n-1}) + w^n \in I$, where the $a_i$'s are  positive integers, and $v = u^n + a_1u^{n-1}w + a_2u^{n-2}w^2 + \cdots + a_{n-1}uw^{n-1}  \in U(R)$. Hence $x^n = (eu + w)^n = ev + w^n$ with $w ^n \in nil(R) \subseteq I$. Thus $ev = x^n - w^n  \in I$, and hence $eu = (ev)(v^{-1}u) \in I$. Thus $x = eu + w \in I$; so $I$ is a prime ideal of $R$. 
\end{proof}

\begin{cor}\label{S3T1}
Let $R$ be a commutative  von-Neumann regular ring. Then a proper ideal $I$ of $R$ is an $n$-semiprimary ideal of  $R$ if and only if $I$ is a prime ideal of $R$.
\end{cor}

\begin{proof}
A commutative ring $R$ is von Neumann regular if and only if $nil(R) = \{0\}$ and $dim(R) = 0$ \cite[page 5]{H}. 
\end{proof}

However, if $I$ is an $n$-semiprimary ideal of  a zero-dimensional commutative ring $R$ for some integer $n\geq 2$ and $nil(R) \nsubseteq I$, then $I$ need not be a prime ideal of $R$. We have the following example.

\begin{exa}
{\rm Let $R = \mathbb{Z}_4\times \mathbb{Z}_2$. Then $dim(R) = 0$ and $I = \{0\}\times \mathbb{Z}_2$ is a $2$-semiprimary ideal of  $R$ with $nil(R) = \{0, 2 \} \times \{0\} \nsubseteq I$. However, $I$ is not a prime ideal of $R$.}
\end{exa}

It is easy to determine the $n$-semiprimary ideals in a Dedekind domain $R$ since every nonzero proper ideal of $R$ is (uniquely) a product of prime (maximal) ideals \cite[Theorem 6.16]{LM}. 

\begin{thm}\label{S3T3.5}
Let $I$ be a nonzero proper ideal of a Dedekind domain $R$. Then $I$ is an $n$-semiprimary ideal of $R$ if and only if $I = P^k$, where $P = \sqrt{I}$ is a prime (maximal) ideal of $R$ and $n \geq k$. Moreover, $\delta_R(I) = n$ if and only if $I = P^n$.
\end{thm}

\begin{proof}
Let $I$ be a nonzero proper ideal of a Dedekind domain $R$. Then $\sqrt{I} =P$ is a prime (maximal) ideal if and only if $I = P^k$ for some positive integer $k$. Thus by Theorem~\ref{T1} and Theorem~\ref{T2}, $I$ is $n$-semiprimary if and only if $I = P^k$ for some positive integer $k$, where  $n \geq k$. The ``in particular'' statement is clear.
\end{proof}

Next, we give a characterization of Dedekind domains in terms of $2$-semiprimary ideals.

\begin{thm} \label{S3T3}
Let $R$ be a Noetherian integral domain. Then the following statements are equivalent.
\begin{enumerate}
\item $R$ is a Dedekind domain.
\item If $I$ is an ideal of $R$ with $\delta_R(I) =2$, then $I = M^2$ for some maximal ideal $M$ of $R$.
\end{enumerate}
\end{thm}

\begin{proof}

{\bf $(1) \Rightarrow (2)$} This follows directly from Theorem~\ref{S3T3.5}.

{\bf $(2) \Rightarrow (1)$}  Let $I$ be an ideal of $R$ with $M^2 \subseteq I \subsetneq M$ for a maximal ideal $M$ of $R$. Then $I$ is $2$-semiprimary  by Theorem~\ref{T2} and not prime (maximal); so $\delta_R(I) = 2$. Thus $I = M^2$ by hypothesis. Hence there are no ideals of $R$ strictly between $M$ and $M^2$ for every maximal ideal $M$ of $R$; so $R$ is a Dedekind domain by \cite[Theorem 6.20]{LM}.
\end{proof}

It is also easy to describe the $n$-semiprimary ideals in a valuation domain. Recall that every proper ideal in a valuation domain is semiprimary \cite[Theorem 17.1(2)]{G}.

\begin{thm}  \label{S3T3.8}
Let $I$ be a proper ideal of a valuation domain $R$ with $P = \sqrt{I}$.

{\rm (a)}   $I$ is an $n$-semiprimary ideal of $R$ if and only if $P^n \subseteq I$.

{\rm (b)} If $P$ is idempotent, then $I$ is an $n$-semiprimary ideal of $R$ if and only if $I = P$.

{\rm (c)} If $P$ is not idempotent, then $I$ is an $n$-semiprimary ideal of $R$ for some positive integer $n$. Moreover, every ideal of $R$ between $P$ and the prime ideal directly below $P$ is an $n$-semiprimary ideal for some positive integer $n$.
\end{thm}

\begin{proof}
(a) If $P^n \subseteq I$, then $I$ is $n$-semiprimary by Theorem~\ref{T2}. Conversely, suppose that $I$ is $n$-semiprimary. Then $x^n \in I$ for every $x \in P$ by Theorem~\ref{T1}; so $P^n = \{r x^n \mid r \in R, x \in P\}  \subseteq I$ {\rm (}cf. \cite[Proposition 2.1 and Corollary 2.2]{AD2}{\rm )}.

(b) This follows directly from part (a).

(c) If $P = \sqrt{I}$ is not idempotent, then $P^n \subseteq I$ for some positive integer $n$ \cite[Theorem 17.1(5)]{G}, and  thus $I$ is $n$-semiprimary by Theorem~\ref{T2}. For the ``moreover'' statement, $P^n \subseteq I$ for some positive integer $n$ since the prime ideal directly below $P$ is $Q = \cap_{n=1}^{\infty}P^n$ \cite[Theorem 17.1(3)(4)]{G}.
\end{proof}

The following example illustrates the possible behavior of $n$-semiprimary ideals in valuation domains $R$ with $dim(R) \leq 2$. The details follow directly from Theorem~\ref{S3T3.8} and well-known facts about the value group of a valuation domain (cf. \cite[Chapter 3]{G}). It is interesting to compare Theorem~\ref{S3T3.8} (resp., Example~\ref{E3.7}) with \cite[Theorem 5.5]{AB1} (resp., \cite[Example 5.6]{AB1}) which concerns $n$-absorbing ideals in a valuation domain. There are $n$-semiprimary ideals that are not $n$-absorbing ideals in some valuation domains $R$ since $I$ is an $n$-semiprimary (resp., $n$-absorbing) ideal of a valuation domain $R$ if and only if $P^n \subseteq I$ (resp., $P^n = I$).

\begin{exa} \label{E3.7}
{\rm  (a) Let $R$ be a one-dimensional valuation domain with maximal ideal $M$. If $M$ is principal, then $R$ is a DVR, and thus every proper ideal of $R$ is an $n$-semiprimary ideal for some positive integer $n$. If $M$ is not principal, then $M^2 = M$, and hence $\{0\}$ and $M$ are the only proper ideals of $R$ that are $n$-semiprimary for some positive integer $n$.

(b) Let $R$ be a two-dimensional valuation domain with prime ideals $\{0\} \subsetneq P \subsetneq M$ and value group $G$. If $G = \mathbb{Z} \oplus \mathbb{Z}$ (all direct sums have the lexicographic order), then $P^2 \neq P$ and $M^2 \neq M$; so every proper ideal of $R$ is $n$-semiprimary for some positive integer $n$.  If $G = \mathbb{Q} \oplus \mathbb{Q}$, then $P^2 = P$ and $M^2 = M$; so $\{0\}$, $P$, and $M$ are the only ideals of $R$ that are $n$-semiprimary for some positive integer $n$. If $G = \mathbb{Z} \oplus \mathbb{Q}$, then $P^2 \neq P$ and $M^2 = M$; so $M$ and every ideal of $R$ contained in $P$ is $n$-semiprimary for some positive integer $n$, but no ideal properly between $P$ and $M$ is $n$-semiprimary for any positive integer $n$. If $G = \mathbb{Q} \oplus \mathbb{Z}$, then $P^2 = P$ and $M^2 \neq M$; so every ideal of $R$ between $P$ and $M$ is $n$-semiprimary for some positive integer $n$, but  $\{0\}$ and $P$ are the only ideals of $R$ contained in $P$ that are $n$-semiprimary for some positive integer $n$.}
\end{exa}

We end this section with two results on idealization. Let $M$ be an $R$-module over a commutative ring $R$. The {\it idealization} of $M$ is the commutative ring $R(+)M = R \times M$ with addition and multiplication defined by $(a,m) + (b,n) = (a+b, m+n)$ and $(a,m)(b,n) = (ab, bm + an)$, respectively, and identity $(1,0)$ (cf. \cite{A?}, \cite[Section 25]{H}). Note that $(\{0\}(+)M)^2 =\{0\}$; so $\{0\}(+)M \subseteq nil(R(+)M)$.

\begin{thm}\label{S3T4}
Let $I$ be a proper ideal of a commutative ring $R$, $M$ an $R$-module, and $S = IM$ a submodule of $M$. If $I$ is an $n$-semiprimary ideal of  $R$, then $I (+) S$ is an $(n+1)$-semiprimary ideal of  $R(+)M$. Moreover, if $I (+) S$ is an $n$-semiprimary ideal of $R(+)M$, then $I$ is an $n$-semiprimary ideal of  $R$.
\end{thm}

\begin{proof}
Let $I$ be an $n$-semiprimary ideal of  $R$ and $(a, m)^{n+1}(b,h)^{n+1} = (a^{n+1}b^{n+1},z)  \in I (+) S$ for $(a, m), (b, h) \in R(+)M$.  Then $a^n \in I$ or $b^n \in I$ by Theorem~\ref{T3.5}(a) since $I$ is an $n$-semiprimary ideal of  $R$. We may assume that $a^n \in I$; so $(n+1)a^nm \in IM = S$. Thus $(a, m)^{n+1} = (a^{n+1}, (n+1)a^nm) \in I (+) S$; so $I (+) S$ is an $(n+1)$-semiprimary ideal of  $R(+)M$. The  ``moreover'' statement is clear.
\end{proof}

\begin{thm}\label{S3T5}
Let  $I$ a proper ideal of a commutative ring $R$ with $char(R) = n \geq 2$, $M$  an $R$-module, and $S$  a submodule of $M$. Then $I (+) S$ is an $n$-semiprimary ideal of  $R(+)M$ if and only if $I$ is 
an $n$-semiprimary ideal of $R$.
\end{thm}

\begin{proof}
If $J = I (+) S$ is an $n$-semiprimary ideal of  $A = R(+)M$, then clearly $I$ is an $n$-semiprimary ideal of  $R$. Conversely,  assume that $I$ is an $n$-semiprimary ideal of  $R$. Let $(a, m)^n(b,h)^n = (a^nb^n,z) \in J$ for  $(a, m), (b, h) \in A$. Then $a^n \in I$ or $b^n \in I$ since $I$ is an $n$-semiprimary ideal of  $R$;  assume that $a^n \in I$. Since $char(R) = n \geq 2$, we have  $na^{n-1}m = 0 \in S$. Thus $(a, m)^n = (a^n, na^{n-1}m) = (a^n, 0) \in J$; so $J$ is an $n$-semiprimary ideal of $A$.
\end{proof}

\section{$n$-powerful semiprimary ideals and $n$-PVDs}

In this section, we study $n$-powerful semiprimary ideals in integral domains and two generalizations of valuation domains, namely, $n$-pseudo valuation domains ($n$-PVDs) and $n$-valuation domains ($n$-VDs).

Recall \cite{BH} (resp., \cite{HH1})  that a proper ideal $I$ of an integral domain $R$ with quotient field $K$ is {\it powerful} (resp., {\it strongly prime}) if whenever $xy \in I$ for $x,y \in K$, then $x \in R$ or $y \in R$ (resp., $x \in I$ or $y \in I$). We begin with an ``$n$'' generalization.

\begin{dfn}
{\rm  Let $R$ be an integral domain with quotient field $K$ and $n$  a positive integer. A proper ideal $I$ of $R$ is an {\it $n$-powerful ideal} of $R$ if whenever $x^ny^n \in I$ for $x, y \in K$, then $x^n \in R$ or $y^n \in R$; and $I$ is an {\it $n$-powerful semiprimary ideal} of $R$ if
whenever $x^ny^n \in I$ for $x, y \in K$, then $x^n \in I$ or $y^n \in I$.}
\end{dfn}

Thus a $1$-powerful (resp., $1$-powerful semiprimary)  ideal is just a powerful (resp., strongly prime) ideal, and an $n$-powerful (resp., $n$-powerful semiprimary) ideal is also an $mn$-powerful (resp., $mn$-powerful semiprimary) ideal for every positive integer $m$. It is well known that prime ideals in a valuation domain are strongly prime ideals. From this observation, it easily follows that $n$-semiprimary ideals in a valuation domain are also $n$-powerful semiprimary ideals;  so Theorem~\ref{S3T3.8}  and Example~\ref{E3.7} also hold for $n$-powerful semiprimary ideals.  However, an $n$-semiprimary ideal need not be an $n$-powerful semiprimary ideal. For example, let $R = \mathbb{Z}_2[[X^2,X^3]]$. Then its maximal ideal $M = (X^2,X^3)$ is a prime ($1$-semiprimary) ideal, but not a strongly prime ($1$-powerful semiprimary) ideal. Also, see Example~\ref{e4.1} for a $2$-semiprimary ideal that is not $2$-powerful semiprimary.

We next give a stronger result.

\begin{thm}\label{S4T0}
Let $R$ be an integral domain with quotient field $K$.

{\rm (a)} Let $I$ be an $n$-semiprimary ideal of $R$. If $\sqrt{I}$ is a strongly prime ideal of $R$, then $I$ is an $n$-powerful semiprimary ideal of $R$.

{\rm (b)}  Let $I \subseteq J$ be proper ideals of $R$. If $J$ is an $n$-powerful ideal of $R$, then $I$ is an $n$-powerful ideal of $R$.

{\rm (c)} Let $I$ be an $n$-powerful {\rm  (}resp.,$n$-powerful semiprimary{\rm )} ideal of $R$ and $S$ a multiplicatively closed subset of $R$ with $I \cap S = \emptyset$. Then $I_S$ is an $n$-powerful {\rm (}resp., $n$-powerful semiprimary{\rm )} ideal of $R_S$.
\end{thm}

\begin{proof}
(a)  Let $P = \sqrt{I}$ and $x^ny^n \in I \subseteq P$ for $x, y \in K$. Then $x \in P$ or $y \in P$ since $P$ is a strongly prime ideal of $R$. Thus $x^n \in I$ or $y^n \in I$ by Theorem~\ref{T1}; so $I$ is an $n$-powerful semiprimary ideal of $R$.

(b) Let $x^ny^n \in I \subseteq J$ for $x, y \in K$. Then $x^n \in R$ or $y^n \in R$ since $J$ is an $n$-powerful 
deal of $R$.  Thus $I$ is an $n$-powerful ideal of $R$.

(c)  This follows easily from the definitions.
\end{proof}

Note that every ideal in a valuation domain is powerful; so an $n$-powerful ideal need not be $n$-powerful 
semiprimary. However, for prime ideals, these two concepts coincide.

\begin{thm}\label{S4T1}
Let $I$ be a prime ideal of an integral domain $R$ with quotient field $K$. Then $I$ is an $n$-powerful semiprimary ideal of $R$ if and only if $I$ is an $n$-powerful ideal of $R$.
\end{thm}

\begin{proof}
If $I$ is an $n$-powerful semiprimary ideal of $R$, then $I$ is certainly an $n$-powerful ideal. Conversely, assume that $I$ is an $n$-powerful prime ideal of $R$. Let $x^ny^n \in I$ for $x, y \in K$. First, suppose that $x^n, y^n \in R$. Since $I$ is a prime ideal of $R$, then $x^n \in I$ or $y^n \in I$.  Thus we may assume that $x^n \not \in R$, and hence  $y^n \in R$ since $I$ is an $n$-powerful ideal of $R$. Since $x^n \not \in R$ and $I$ is an $n$-powerful ideal of $R$, we have $x^{2n} = x^nx^n \not \in I$. Assume that $x^{2n} \in R$; so $y^{2n}, x^{2n} \in R$. Since  $x^{2n}y^{2n} \in I$ and $x^{2n} \not \in I$, we have $y^{2n} \in I$. Since $y^n \in R$, $I$ is a prime ideal of $R$, and $ y^ny^n = y^{2n} \in I$, we have $y^n \in I$. Now, assume that $x^{2n} \not \in R$. Since $(y^{2}/xy)^n(x^2)^n = (y^{2n}/x^ny^n)x^{2n} = x^ny^n \in I$, $x^{2n} \not \in R$, and $I$ is an $n$-powerful ideal of $R$, we have $y^{2n}/x^ny^n \in R$. Thus $y^{2n} = x^ny^n(y^{2n}/x^ny^n) \in I$. Since $y^n \in R$, $I$ is a prime ideal of $R$, and $y^{2n} = y^ny^n \in I$, we have $y^n \in I$. Hence $I$ is an $n$-powerful semiprimary ideal of $R$.
\end{proof}

\begin{thm}\label{S4T3}
Let $P \subseteq Q$ be prime ideals of an integral domain $R$. If $Q$ is an $n$-powerful semiprimary ideal of $R$, then $P$ is an $n$-powerful semiprimary ideal of $R$.
\end{thm}

\begin{proof}
Let $Q$ be an $n$-powerful ideal of $R$; so $P$ is an $n$-powerful ideal of $R$ by Theorem \ref{S4T0}(b). Thus $P$ is an $n$-powerful semiprimary ideal of $R$ by Theorem \ref{S4T1}.
\end{proof}

Let $I$ be a proper ideal of an integral domain $R$. As the ``powerful'' analogs of $W_R(I)$ and $\delta_R(I)$, we define $\overline{W}_R(I) = \{ n \in \mathbb{N} \mid  I$ is an $n$-powerful semiprimary ideal of $R \}$ and $\overline{\delta}_R(I) = min\overline{W}_R(I)$ (let $\overline{\delta}_R(I) = \infty$ if $\overline{W}_R(I) = \emptyset$). Note that $\overline{W}_R(I) \subseteq W_R(I)$ and $\delta_R(I) \leq \overline{\delta}_R(I)$. The next example shows that the analogs of Theorem~\ref{T3.5}(b) and Theorem~\ref{S4T0}(b) do not hold for $n$-powerful semiprimary ideals. In particular, if $I$ is an $n$-powerful semiprimary ideal, then $I$ is an $n$-semiprimary ideal. Thus $I$ is also an $m$-semiprimary ideal for every integer $m \geq n$, but $I$ need not be an $m$-poweful semiprimary ideal.

\begin{exa}  \label{e4.1}
{\rm Let $R = F[[X^2,X^5]] = F + FX^2 + X^4F[[X]]$, where $F$ is a field. Then $R$ is quasilocal with maximal ideal $M = (X^2,X^5) = FX^2 + X^4F[[X]]$ and quotient field $K = F[[X]][1/X]$. Clearly $M$ is a $2$-semiprimary ideal of $R$, but not a $3$-powerful semiprimary ideal of $R$ since $X^3X^3 = X^6\in M$, but $X^3 \not\in M$. Moreover,  $M$ is a $2$-powerful semiprimary ideal of $R$ if and only if $char(F) = 2$, and $M$ is an $n$-powerful semiprimary ideal of $R$ for every integer $n \ge 4$. So, for $R = \mathbb{Z}_2[[X^2,X^5]]$, $M$ is a $2$-powerful semiprimary ideal, but not a $3$-powerful semiprimary ideal, and $\overline{W}_R(M) = \mathbb{N} \setminus \{1,3\}$. Thus the ``powerful'' analog of Theorem~\ref{T3.5}(b) fails for $M$. Let $I = X^4F[[X]]$. Then $I$ is a $2$-semiprimary ideal of $R$, but not a $2$-powerful semiprimary ideal of $R$ since $X^2X^2 \in I$, but $X^2 \not\in I$. So the ``semiprimary'' analog of Theorem~\ref{S4T0}(b) fails for $I  \subseteq  J = M$ when $char(F) = 2$.}
\end{exa}

Recall \cite{HH1} that an integral domain $R$ is a {\it pseudo-valuation domain} (PVD) if every prime ideal of $R$ is strongly prime.  A PVD is neccessarily quasilocal \cite[Corollary 1.3]{HH1}. A quasilocal integral domain $R$ with maximal ideal $M$ is a PVD $\Leftrightarrow$ $M$ is strongly prime \cite[Theorem 1.4]{HH1}, and $R$ is a PVD $\Leftrightarrow$ $(M:M)$ is a valuation domain with maximal ideal $M$ \cite[Proposition 2.5]{AD}. Let $T = K + M$ be a valuation domain, where $K$ is a field and $M$ is the maximal ideal of $T$. Then for a proper subfield $k$ of $K$, the subring $R = k + M$ is a PVD which is not a valuation domain \cite[Example 2.1]{HH1}. By Theorem~\ref{S4T0}(a), every $n$-semiprimary ideal in a PVD is an $n$-powerful semiprimary ideal.

We next give an ``$n$'' generalization of PVDs.

\begin{dfn}
{\rm Let $R$ be an integral domain and $n$ a positive integer. Then $R$ is an {\it $n$-pseudo-valuation domain {\rm (}$n$-PVD{\rm )}} if every prime ideal of $R$ is an $n$-powerful semiprimary ideal of $R$.}
\end{dfn}

Note that a $1$-PVD is just a PVD and an $n$-PVD is also an $mn$-PVD for every positive integer $m$. The next several results show that $n$-PVDs behave very much like PVDs (cf. \cite{AA}, \cite{A}, \cite{A2}, \cite{AD}, \cite{BH}, \cite{HH1}, and \cite{HH2}).

\begin{thm}\label{S4T2}
Let $R$ be an $n$-PVD. Then $R$ is quasilocal.
\end{thm}

\begin{proof}
By way of contradiction, assume that $M$ and $N$ are distinct maximal ideals of $R$. Let $x \in M\setminus N$ and $y \in N\setminus M$. Then $(x/y)^n(y^2)^n = (x^n/y^n)y^{2n} = x^ny^n \in M$, and thus  $(x/y)^n  \in M$ since $M$ is an $n$-powerful semiprimary ideal of $R$ and $(y^2)^n \not \in M$. Hence $x^n = (x/y)^ny^n \in N$; so $x \in N$,  a contradiction. Thus $R$ is quasilocal.
\end{proof}

In view of Theorem \ref{S4T1}, Theorem \ref{S4T3}, and the proof of Theorem \ref{S4T2}, we have the following result.

\begin{cor}\label{S4C1}
An integral domain $R$ is an $n$-PVD if and only if some maximal ideal of $R$ is an $n$-powerful semiprimary ideal of $R$, if and only if some maximal ideal of $R$ is an $n$-powerful ideal of $R$.
\end{cor}

Recall (\cite{Do}, \cite {B?})  that a prime ideal $P$ of a commutative ring $R$ is a {\it divided prime ideal}  of $R$ if $x \, | \,  p$ (in $R$) for every $x \in R \setminus P$ and $p \in P$ (i.e., $(x)$ is comparable to $P$ for every $x \in R$), and $R$ is a {\it divided ring} if every prime ideal of $R$ is divided. We next give the ``n'' generalization.

\begin{dfn}
{\rm Let $R$ be a commutative ring and $n$  a positive integer. Then a prime ideal $P$ of $R$ is an {\it $n$-divided prime ideal} of $R$ if $x^n \, | \, p^n$ (in $R$) for every $x \in R\setminus P$ and $p \in P$. Moreover,  $R$ is an {\it $n$-divided ring} if every prime ideal of $R$ is an $n$-divided prime ideal of $R$.}
\end{dfn}

A $1$-divided prime ideal (resp., ring) is just a divided prime ideal (resp., ring), and an $n$-divided prime ideal is $mn$-divided for every positive integer $m$. Thus an $n$-divided ring is $mn$-divided for every positive integer $m$. 

The next several results show that $n$-divided rings behave very much like divided rings (cf. \cite{B?}, \cite{Do}).

\begin{thm} \label{S4T3.5}
Let $R$ be an $n$-divided commutative ring. Then the set of prime ideals of $R$ is linearly orderd by inclusion. In particular, $R$ is quasilocal.
\end{thm}

\begin{proof}
Let $P$ and $Q$ be prime ideals of an $n$-divided commutative ring $R$ with $P \not \subseteq Q$. We show that $Q \subseteq P$. Let $x \in P \setminus Q$; then $x^n \, | \,  q^n$ for every $q \in Q$ since $Q$ is a divided prime ideal of $R$. Thus $q^n \in (x^n) \subseteq P$; so $q \in P$ for every $q \in Q$. Hence $Q \subseteq P$.
\end{proof} 

\begin{thm}\label{S4T4}
Let $P$ a prime ideal of an integral domain $R$. If $P$ is an $n$-powerful semiprimary ideal of $R$, then $P$ is an $n$-divided prime ideal of $R$. Moreover, the set of prime ideals of $R$ that are contained in $P$ is linearly ordered by inclusion.
\end{thm}

\begin{proof}
Let $x \in R\setminus P$ and $p \in P$. Then $(p/x)^nx^n = (p^n/x^n)x^n = p^n \in P$. Thus $p^n/x^n \in P$ since $x^n \not \in P$ and $P$ is an $n$-powerful semiprimary ideal of $R$. Hence $p^n = (p^n/x^n)x^n$; so $x^n \, | \, p^n$ (in $R$). Thus $P$ is an $n$-divided prime ideal of $R$. Now suppose that $F$ and $H$ are distinct prime ideals of $R$ contained in $P$. Then $F$ and $H$ are $n$-powerful semiprimary ideals of $R$ by Theorem \ref{S4T3}, and hence are $n$-divided prime ideals. The proof of Theorem~\ref{S4T3.5} shows that $F$ and $H$ are comparable under inclusion.
\end{proof}

\begin{cor}\label{S4C2}
Let $R$ be an $n$-PVD. Then $R$ is an $n$-divided domain and the set of prime ideals of $R$ is linearly ordered by inclusion. Moreover, if $R$ is Noetherian, then $dim(R) \leq 1$.
\end{cor}

\begin{proof}
We need only prove the ``moreover'' statement; it follows directly from \cite[Theorem 144]{K}.
\end{proof}

Let $R$ be an integral domain with quotient field $K$, $S \subseteq R$, and $n$ a positive integer. Define $E_n(S) = \{ x  \mid x^n \not \in S, x \in K \}$ and $A_n(S) = \{ x^n  \mid x^n \in S, x \in K \}$. We next use these two sets to give another characterization of $n$-powerful semiprimary ideals. Note that actually $x^{-n}d \in A_n(P)$ in Theorem~\ref{S4T5} and Corollary~\ref{S4C3}(4), and $x^{-n}d \in A_n(M)$ in Corollary~\ref{S4C4}(3).

\begin{thm} \label{S4T5}
Let $P$  a prime ideal of an integral domain $R$ with quotient field $K$. Then $P$ is an $n$-powerful semiprimary ideal of $R$ if and only if $x^{-n}d \in P$ for every $x \in E_n(P)$ and $d \in A_n(P)$.
\end{thm}

\begin{proof}
Suppose that $x^{-n}d \in P$ for every $x \in E_n(P)$ and  $d \in A_n(P)$. Let $x^ny^n \in P$ for $x, y \in K$ with $x^n \not \in P$; so $x\in E_n(P)$. Since $x^ny^n = (xy)^n  \in A_n(P)$, we have $y^n = x^{-n}(x^ny^n) \in P$. Thus $P$ is an $n$-powerful semiprimary ideal of $R$.

 Conversely, suppose that $P$ is an $n$-powerful semiprimary ideal of $R$. Let $d\in A_n(P)$; so $d = a^n \in P$ for some $a\in K$ and $ x^n(x^{-1}a)^n = x^nx^{-n}a^n =a^n \in P$ for every $0 \neq x \in K$. Suppose that $x\in E_n(P)$. Then $x^n \not \in P$; so $(x^{-1}a)^n \in P$  since $P$ is an $n$-powerful semiprimary ideal of $R$. Thus $x^{-n}d =   x^{-n}a^n = (x^{-1}a)^n \in P$.
\end{proof}

The proof of the following result is similar to that  of Theorem \ref{S4T5}, and thus will be omitted.

\begin{thm}\label{S4T6}
Let $I$  a proper ideal of an integral domain $R$. Then $I$ is an $n$-powerful ideal of $R$ if and only if $x^{-n}d \in R$ for every $x \in E_n(R)$ and $d \in A_n(I)$.
\end{thm}

In view of Theorem \ref{S4T1}, Theorem \ref{S4T5}, and Theorem \ref{S4T6}, we have the following result.

\begin{cor}\label{S4C3}
Let $P$ be a prime ideal of an integral domain $R$. Then the following statements are equivalent.
\begin{enumerate}
\item $P$ is an $n$-powerful semiprimary ideal of $R$.
\item $P$ is an $n$-powerful ideal of $R$.
\item $x^{-n}d \in R$ for every $x \in E_n(R)$ and $d \in A_n(P)$.
\item $x^{-n}d \in P$ for every $x \in E_n(P)$ and $d \in A_n(P)$.
\end{enumerate}
\end{cor}

In view of Corollary \ref{S4C1}, Theorem \ref{S4T5}, and Theorem \ref{S4T6}, we have the following result.

\begin{cor}\label{S4C4}
Let $R$ be a quasilocal integral domain with maximal ideal $M$. Then the following statements are equivalent.
\begin{enumerate}
\item $R$ is an $n$-PVD.
\item $x^{-n}d \in R$ for every $x \in E_n(R)$ and $d \in A_n(M)$.
\item $x^{-n}d \in M$ for every $x \in E_n(M)$ and $d \in A_n(M)$.
\end{enumerate}
\end{cor}

If $R$ is a PVD, then $R/P$ is also a PVD for $P$ a prime ideal of $R$ \cite[Lemma 4.5(i)]{Do2}. The analogous result holds for $n$-PVDs.

\begin{thm}
Let $P$ be a prime ideal of an $n$-PVD $R$. Then $R/P$ is an $n$-PVD.
\end{thm}

\begin{proof}
Let $M$ be the maximal ideal of $R$, $K$ the quotient field of $R$, $F = R_P/PR_P$  the quotient field of $A = R/P$, and $H_n(M/P) = \{x^n \in M/P \mid x \in F\}$. Suppose that $x = a + P, y = b + P \in A$, and $x^n \nmid y^n$ in $A$. Then $a^n \nmid b^n$ in $R$; so $b^n \mid a^nd$ in $R$ for every $d \in A_n(M)$ by Corollary~\ref{S4C4}. Thus $y^n \mid x^nh$ in $A$ for every $h \in H_n(M/P)$; so $A$ is an $n$-PVD by Corollary~\ref{S4C4} again.
\end{proof}

Let $n$ be a positive integer. Recall that an integral domain $R$ with quotient field $K$ is  {\it $n$-root closed} if whenever $x^n \in R$ for $x \in K$, then $x \in R$; and $R$ is {\it root closed} if $R$ is $n$-root closed for every positive integer $n$. For example, an integrally closed integral domain is root closed. Note that $R$ is $mn$-root closed if and only if $R$ is $m$-root closed and $n$-root closed. Thus $\mathcal{C}(R) = \{ n \in \mathbb{N} \mid R$ is $n$-root closed$\}$ is a multiplicative submonoid on $\mathbb{N}$ generated by some set of prime numbers. Moreover, for $S$ any multiplicative submonoid of $\mathbb{N}$ generated by a set of prime numbers, $S = \mathcal{C}(R)$ for some integral domain $R$ \cite[Theorem 2.1]{A3}. 

For $n$-root closed integral domains, the $n$-PVD and PVD concepts coincide.

\begin{thm}\label{S4T7}
Let $R$ be an $n$-root closed integral domain with quotient field $K$. Then $R$ is an $n$-PVD if and only if $R$ is a PVD. In particular, an integrally closed $n$-PVD is a PVD.
\end{thm}

\begin{proof}
If $R$ is a PVD, then clearly $R$ is an $n$-PVD. Conversely, let $R$ be an $n$-root closed $n$-PVD with maximal ideal $M$. We show that $M$ is a powerful ideal of $R$. Let $xy \in M$ for $x, y\in K$ and $x\not \in R$. Then $x^ny^n \in M$ and  $x^n \not \in R$ since $R$ is $n$-root closed. Thus $y^n \in M \subseteq  R$ since $M$ is an $n$-powerful semiprimary ideal of $R$, and hence $y \in R$ since $R$ is $n$-root closed. Thus $M$ is a powerful ideal of $R$; so $M$ is a strongly prime ideal of  R (i.e., $M$ is a $1$-powerful semiprimary ideal of  $R$) by Theorem \ref{S4T1}. Hence $R$ is a PVD. The ``in particular'' statement is clear.
\end{proof}

Recall (\cite{AZ1}, \cite{AKL}, \cite{AZ2}, \cite{MMM}) that an integral domain $R$ with quotient field $K$ is an {\it almost valuation domain} if for every  $0 \neq x \in K$, there is a positive integer $n$ (depending on $x$) such that $x^n \in R$ or $x^{-n} \in R$. We have the following ``n'' generalization.

\begin{dfn}
{\rm Let $n$ be a positive integer. An integral domain $R$ with quotient field $K$ is an {\it $n$-valuation domain {\rm (}$n$-VD{\rm )}} if  $x^n \in R$ or $x^{-n} \in R$ for every $0 \neq x \in K$.}
\end{dfn}

It is clear that a valuation domain is an $n$-VD for every positive integer $n$, an $n$-root closed $n$-VD is a valuation domain, an $n$-VD is an almost valuation domain, an $n$-VD is also an $mn$-VD for every positive integer $m$, and an $n$-VD is an $n$-PVD. Moreover, an $n$-VD is quasilocal, an overring of an $n$-VD is also an $n$-VD, and a Noetherian $n$-VD has  (Krull) dimension at most one. 

We have the following elementary results about $n$-VDs which show that $n$-VDs behave very much like valuation domains (cf. \cite[Chapter 3]{G}). In \cite[page 3]{AA}, it was observed that $R$ is a valuation domain if and only if $R$ is a strongly prime ideal of $R$ (here, and in Theorem~\ref{nvd}(a)(5), we drop the usual assumption that a prime ideal is a proper ideal).

\begin{thm} \label{nvd}
Let $R$ be an integral domain with quotient field $K$ and $n$ a positive integer.

{\rm (a)} The following statements are equivalent.

\begin{enumerate}
   \item $R$ is an $n$-VD.
   \item $x^n  \, | \, y^n$ or $y^n \, | \, x^n$ for every $0 \neq x, y \in K$.
   \item  $x^n \, | \, y^n$ or $y^n \, | \, x^n$ for every $0 \neq x, y \in R$.
   \item Let $G$ be the group of divisibility of $R$. Then  for every $g \in G$, either $ng \geq 0$ or $ng < 0$.
   \item  $R$ is an $n$-powerful semiprimary ideal of $R$.
\end{enumerate} 

{\rm (b)} Let $R$ be an $n$-VD. Then $R$ is an $n$-divided domain, and thus the prime ideals of $R$ are linearly ordered by inclusion.

{\rm (c)} Let $R$ be an $n$-VD and $x \in K$. If $x^n$ is integral over $R$, then $x^n \in R$.
\end{thm}

\begin{proof}
The proofs are essentially the same as for valuation domains. See \cite[Theorem 16.3]{G} for part (a) and \cite[Theorem 17.5]{G} for part (c).  Part (b) follows from Corollary~\ref{S4C2} since an $n$-VD is also an $n$-PVD.
\end{proof}

An $n$-VD is always an $n$-PVD, but an $n$-PVD need not be an $n$-VD. Also, an almost valuation domain 
need not be an $n$-VD for any positive integer $n$.

\begin{exa} \label{ex22}
{\rm (a)  Let $R = \mathbb{Q} + X\mathbb{R}[[X]]$. Then $R$ is a PVD with maximal ideal $X\mathbb{R}[[X]]$ and quotient field $\mathbb{R}[[X]][1/X]$, and thus $R$ is an $n$-PVD for every positive integer $n$. However, $R$ is not an $n$-VD for any positive integer $n$ since $\pi^n, \pi^{-n} \notin R$ for every positive integer $n$.

(b) Let $R = \mathbb{Z}_p + XF[[X]]$, where $p$ is a positive prime integer and $F = \overline{\mathbb{Z}_p}$ is the algebraic closure of $\mathbb{Z}_p$. Then $R$ is an almost valuation domain with maximal ideal $XF[[X]]$ and quotient field $F[[X]][1/X]$, but not an $n$-VD for any positive integer $n$. This follows from the fact that for every $0 \neq a \in F$, there is a positive integer $n$ such that $a^n = 1$; but for every positive integer $n$, there is a $b \in F$ such that $b^n \notin \mathbb{Z}_p$ and $b^{-n} \notin \mathbb{Z}_p$. Note that $R$ is also a PVD, and thus an $n$-PVD for every positive integer $n$.}
\end{exa}

 In some cases, an overring of an $n$-PVD is also an $n$-VD.

\begin{thm}\label{S4T8}
Let $R$ be an $n$-PVD  with maximal ideal $M$, quotient field $K$, and $V$ an overring of $R$ such that $1/s \in V$ for some $0 \neq s \in M$. Then $V$ is an $n$-VD, and thus $V$ is an almost valuation domain.
\end{thm}

\begin{proof}
Let $x \in K$ with $x^n \notin V$; so $x \in E_n(R)$. Then $x^{-n}d \in M$ for every $d \in A_n(M)$ by Corollary \ref{S4C4}. In particular, $a = x^{-n}s^n \in M$ since $d = s^n \in A_n(M)$, and thus $x^{-n} = a/s^n \in V$ since $1/s \in V$. Hence $V$ is an $n$-VD, and thus $V$ is an almost valuation domain.
\end{proof}

By Theorem~\ref{S4T0}(c), if $R$ is an $n$-PVD, then $R_P$ is also an $n$-PVD for every nonmaximal prime ideal $P$ of $R$. We next give a stronger result; $R_P$ is an $n$-VD.

\begin{thm}\label{S4T9}
Let $R$ be an $n$-PVD  with maximal ideal $M$ and $P \subsetneq M$  a prime ideal of $R$. Then $R_P$ is an $n$-VD, and thus $R_P$ is an almost valuation domain. Moreover, $x^n \in R$ for every $x \in P_P$, and hence $P_P \subsetneq \overline{R}$.
\end{thm}

\begin{proof}
Since $P \subsetneq M$, there is an $s \in M \setminus P$. Thus $1/s \in R_P$; so $R_P$ is an $n$-VD (and hence also  an almost valuation domain) by Theorem \ref{S4T8}. Let $x \in P_P$; so $x = a/s$ for some $a\in P$ and $s \in R\setminus P$. Thus $s^n \, | \, a^n$ (in $R$) since $P$ is an $n$-divided prime ideal of $R$ by Theorem \ref{S4T4}. Hence $x^n = a^n/s^n \in R$; so $P_P \subsetneq \overline{R}$.
\end{proof}

We next show that $n$-divided principal prime ideals are actually maximal ideals.

\begin{thm}\label{S4T10}
Let $R$ be an integral domain $R$ with quotient field $K$ and (nonzero) principal prime ideal $P$. If $P$ is an $n$-divided ideal of $R$, then $P$ is a maximal ideal of $R$. Moreover, if $P$ is also an $n$-powerful semiprimary ideal of $R$, then $P$ is a maximal ideal of $R$ and $R$ is an $n$-VD.
\end{thm}

\begin{proof}
Let $P = (p)$ for a prime element $p$ of $R$. By way of contradiction, assume that $P$ is not a maximal ideal; so there is a nonunit $x \in R\setminus P$. If $P$ is an $n$-divided prime ideal of $R$, then there is a $y \in R$ with $p^n = x^nyp^{n}$ or $p^n = x^nwp^m$ for some positive integer $m < n$ and $w\in R\setminus P$. If $p^n = x^nyp^{n}$, then $1 = x^ny$, and thus $x \in U(R)$, a contradiction. If $p^n = x^nwp^m$, then  $x^nw = p^{n-m} \in P$, which is a contradiction since  $x \notin P$ and $w \notin P$. Hence $P$ is a maximal ideal of $R$. 

Now, suppose that $P = (p)$ is an $n$-powerful semiprimary ideal of $R$. Then $P$ is an $n$-divided prime ideal of $R$ by Theorem \ref{S4T4}. Thus $P$ is a maximal ideal of $R$, and hence $R$ is an $n$-PVD by Corollary \ref{S4C1}. Finally, we show that $R$ is an $n$-VD. Let $x \in K$, and suppose that $x^n \notin R$. Then $x^n \notin P$, and thus $x^{-n}p^n \in P$ by Theorem \ref{S4T5}. Since $x^{-n}p^n \in P = (p)$, we have $x^{-n}p^n = hp^n$ for some $h \in R$ or $x^{-n}p^n = dp^m$ for some positive integer $m < n$ and $d \in U(R)$. If $x^{-n}p^n = dp^m$, then $x^n = d^{-1}p^{n-m} \in R$, a contradiction. Thus $x^{-n}p^n = hp^n$ for some $h \in R$, and hence $x^{-n} = h \in R$. Thus $R$ is an $n$-VD.
\end{proof}

We have already observed several parts of the next theorem. One interesting consequence is that if $P$ is an $n$-powerful semiprimary prime ideal of an integral domain $R$ with quotient field $K$, then $\{ x \in K \mid x^m \in P$ for some positive integer $m \} = \{ x \in K \mid x^n \in P \}$ (cf. Theorem~\ref{T1}).

\begin{thm}\label{S4T11}
Let $P$ be a prime ideal of an integral domain $R$ with quotient field $K$. If $P$ is an $n$-powerful semiprimary ideal of $R$, then $P$ is an $mn$-powerful semiprimary ideal of $R$ for every positive integer $m$. Furthermore, if $x^m \in P$ for a positive integer $m$ and $x \in K$, then $x^n \in P$. In particular, if $R$ is an $n$-PVD, then $R$ is an $mn$-PVD for every positive integer $m$.
\end{thm}

\begin{proof}
Let $m$ be a positive integer. Assume that $x^{mn}y^{mn} \in P$ for  $x, y \in K$. Then $(x^m)^n(y^m)^n \in P$.  Since $P$ is an $n$-powerful semiprimary ideal of $R$, $(x^m)^n = x^{mn} \in P$ or $(y^m)^n = y^{mn} \in P$. Thus $P$ is an $mn$-powerful semiprimary ideal of $R$. Next, assume that $x^m \in P$ for $x \in K$ and some positive integer $m$; so $x^{mn} = (x^m)^n  \in P$. Let $d$ be the least positive integer such that $x^{dn} \in P$. Since $(x^{d-1})^nx^n = x^{dn} \in P$ and $P$ is an $n$-powerful semiprimary ideal of $R$, we have $(x^{d-1})^n \in P$ or $x^n \in P$. Hence $d = 1$, and thus $x^n \in P$. The ``in particular'' statement is clear.
\end{proof}

The next several results concern integral overrings of an $n$-PVD. In particular, an integral overring of an $n$-PVD is an $n$-PVD, and the integral closure of an $n$-PVD is a PVD. Note that $\{ x \in K \mid x^n \in M \} = \{ x \in \overline{R} \mid x^n \in M \}$ in the next several results and $\sqrt{MB} = \sqrt{M\overline{R}} \cap B$ for $B$ an integral overring of $R$.

\begin{thm} \label{S4T12} 
Let $R$ be an $n$-PVD with maximal ideal $M$ and quotient field $K$. If $B$ is an integral overring of $R$, then $B$ is an $n$-PVD with maximal ideal $M_B = \sqrt{MB} = \{x \in B \mid x^n \in M\}$.
\end{thm}

\begin{proof}
Let $m \in M$. Then $\sqrt{mR}$ is a prime ideal of $R$ since the prime ideals of $R$ are linearly ordered (under inclusion) by Corollary~\ref{S4C2}, and thus $\sqrt{mR}$ is an $n$-powerful semiprimary ideal of $R$ since $R$ is an $n$-PVD. We show that $\sqrt{mB}$ is an $n$-powerful semiprimary ideal of $B$ and $\sqrt{mB} = \{ x \in B \mid x^n \in \sqrt{mR} \}$. Let $x^ny^n \in \sqrt{mB}$ for $0 \neq x, y \in K$. Then $x^{nk}y^{nk} = (xy)^{nk} = fm$ for some positive integer $k$ and $0 \neq f \in B$. Note that $f^{-n} \not \in M$; for if $f^{-n} \in M$, then $1/a = f^n \in B$ for some $a \in M$, a contradiction since $B$ is integral over $R$. Then $f^nm^n \in \sqrt{mR}$ since $f^{-n}(fm)^n = m^n \in \sqrt{mR}$, $f^{-n} \not \in \sqrt{mR} \subseteq M$, and $\sqrt{mR}$ is an $n$-powerful semiprimary ideal of $R$. Thus $(x^{nk})^n(y^{nk})^n = (xy)^{nkn} = f^nm^n \in \sqrt{mR}$; so $x^{nkn} \in \sqrt{mR} \subseteq \sqrt{mB}$ or $y^{nkn} \in \sqrt{mR} \subseteq \sqrt{mB}$. Hence $x^n \in \sqrt{mR} \subseteq \sqrt{mB}$ or $y^n \in  \sqrt{mR} \subseteq \sqrt{mB}$ by Theorem~\ref{S4T11}. Thus $\sqrt{mB}$ is an $n$-powerful semiprimary ideal of $B$, and hence a prime ideal of $B$ by Theorem~\ref{T1}. A slight modification of the above proof also shows that $\sqrt{mB} = \{ x \in B \mid x^n \in \sqrt{mR} \}$.

We next show that $M_B = \{x \in B \mid x^n \in M\}$ is an $n$-powerful semiprimary ideal of $B$.  First, we show that $M_B$ is an ideal of $B$. Let $x_1, x_2 \in M_B$; so $x_1^n = m_1 \in M$ and $x_2^n = m_2 \in M$. Thus $x_1 \in \sqrt{m_1B}$ and $x_2 \in \sqrt{m_2B}$. Since the prime ideals of $R$ are linearly ordered, we may assume that $\sqrt{m_1R} \subseteq \sqrt{m_2R}$, and hence $\sqrt{m_1B} \subseteq \sqrt{m_2B}$. Thus $x_1 + x_2 \in \sqrt{m_2B} = \{ x \in B \mid x^n \in \sqrt{m_2R} \} \subseteq M_B$. Next, let $x \in M_B$ and $y \in B$. Then $x^n = m_3 \in M$; so $x \in \sqrt{m_3B}$. Thus $xy \in \sqrt{m_3B} \subseteq M_B$; so $M_B$ is an ideal of $B$.
A similar argument to that for $\sqrt{mB}$ above shows that if $x^ny^n \in M_B$ for $0 \neq x, y \in K$, then $x^n \in \sqrt{mR} \subseteq M \subseteq M_B$ or $y^n \in \sqrt{mR} \subseteq M \subseteq M_B$. Hence $M_B$ is an $n$-powerful semiprimary ideal of $B$, and thus $M_B$ is a prime ideal of $B$ since it is a radical ideal of $B$ by Theorem~\ref{S4T11}. Hence $M_B$ is a maximal ideal of $B$ since $B$ is integral over $R$ and $M_B \cap R = M$; so $B$ is an $n$-PVD by Corollary~\ref{S4C1}. Clearly $M_B = \{x \in B \mid x^n \in M\} \subseteq \sqrt{MB}$, and $\sqrt{MB} \subseteq M_B$ since $MB \subsetneq B$ as $B$ is integral over $R$. Thus $M_B = \sqrt{MB}$.
\end{proof}

\begin{cor}\label{S4T13} 
Let $R$ be an $n$-PVD with maximal ideal $M$ and quotient field $K$. Then $\overline{R}$ is a PVD {\rm (}$1$-PVD{\rm )} with maximal ideal $\sqrt{M\overline{R}} = \{x \in K \mid x^n \in M\}$.
\end{cor}

\begin{proof}
By Theorem \ref{S4T12}, $\overline{R}$ is an $n$-PVD with maximal ideal $\sqrt{M\overline{R}} = 
M_{\overline{R}} = \{ x \in \overline{R} \mid x^n \in M \} = \{ x \in K \mid x^n \in M \}$. Thus $\overline{R}$ is a PVD by Theorem \ref{S4T7}.
\end{proof}

\begin{cor}
Let $P$  be a nonzero finitely generated prime ideal of an $n$-PVD $R$. Then $W = (P : P)$ is an $n$-PVD with maximal ideal $\sqrt{MW} = \{ x \in W \mid x^n \in M \}$. In particular, if $R$ is a Noetherian $n$-PVD with maximal ideal $M$, then $(M : M)$ is an $n$-PVD. 
\end{cor}

\begin{proof}
Note that $W = (P : P)$ is integral over $R$ since $P$ is finitely generated. Thus $W$ is an $n$-PVD with maximal ideal $\sqrt{MW} = \{ x \in W \mid x^n \in M \}$ by 
Theorem~\ref{S4T12}. The ``in particular'' statement is clear. (However, recall that a Noetherian $n$-PVD $R$ has $dim(R) \leq 1$ by Corollary~\ref{S4C2}).
\end{proof}

The converse of Corollary~\ref{S4T13} also holds.

\begin{thm} \label{S4T14}
Let $R$ be a quasilocal integral domain with maximal ideal $M$ and quotient field $K$. Then $R$ is an $n$-PVD  if and only if $\overline{R}$ is a PVD with maximal ideal  $\sqrt{M\overline{R}} = \{x \in K \mid  x^n \in M\}$. 
\end{thm}

\begin{proof}
Let $R$ be an $n$-PVD. Then $\overline{R}$ is a PVD with maximal ideal  $\sqrt{M\overline{R}} = \{x \in K \mid  x^n \in M\}$ by Corollary~\ref{S4T13}. Conversely, suppose that $\overline{R}$ is a PVD with maximal ideal  $N = \sqrt{M\overline{R}} = \{x \in K \mid  x^n \in M\}$. Then $M = R \cap N$ since $M \subseteq N$. Let $x^ny^n = (xy)^n \in M$ for $x, y \in K$; so $xy \in N$. Thus $x \in N$ or $y \in N$ since $N$ is a strongly prime ideal of $\overline{R}$. Hence $x^n \in M$ or $y^n \in M$; so $M$ is an $n$-powerful semiprimary ideal of $R$. Thus $R$ is an $n$-PVD  by Corollary~\ref{S4C1}. 
\end{proof}

\begin{cor}\label{S4C5}
Let $R$ be a quasilocal integral domain with maximal ideal $M$ and  
quotient field $K$. Then the following statements are equivalent.
\begin{enumerate}
\item $R$ is an $n$-PVD.
\item $\overline{R}$ is a PVD with maximal ideal  $\sqrt{M\overline{R}} = \{x \in K \mid  x^n \in M\}$.
\item $N = \sqrt{M\overline{R}} = \{x \in K \mid  x^n \in M\}$ is a maximal ideal of $\overline{R}$ 
such that $(N : N)$ is a valuation domain with maximal ideal $N$.
\end{enumerate}
\end{cor}

\begin{proof}
 {$(1) \Leftrightarrow(2)$}  is  Theorem~\ref{S4T14}, and {$(2) \Leftrightarrow(3)$}  is clear by \cite[Proposition 2.5]{AD}.
\end{proof}

We have seen that integral overrings of an $n$-PVD are also $n$-PVDs. We next determine when every overring of an $n$-PVD is an $n$-PVD. Note that an integrally closed PVD need not be a valuation domain. For example, $R = \mathbb{Q} + X\mathbb{C}[[X]]$ is a PVD, and $\overline{R} = \overline{\mathbb{Q}} + X\mathbb{C}[[X]]$ is a PVD, but not a valuation domain, where $\overline{\mathbb{Q}}$ is the algebraic closure of $\mathbb{Q}$. In this case, $\mathbb{Q}[\pi] + X\mathbb{C}[[X]]$ is a (non-integral) overring of $R$ which is not an $n$-VD or $n$-PVD for any positive integer $n$.
\begin{thm}\label{S4T15}
Let $R$ be an $n$-PVD with maximal ideal $M$. Then every overring of $R$ is an $n$-PVD if and only if $\overline{R}$ is a valuation domain. Moreover, if $\overline{R}$ is a valuation domain, then every non-integral overring of $R$ is an $n$-VD.
\end{thm}

\begin{proof}
Suppose that every overring of $R$ is an $n$-PVD. Since $\overline{R}$ is a PVD by Theorem~\ref{S4T7}, the proof of \cite[Proposition 2.7]{HH2} shows that if $\overline{R}$ is not a valuation domain, then there is a non-quasilocal overring $B$ of $\overline{R}$ (and hence $B$ is an overring of $R$). Thus $B$ cannot be an $n$-PVD by Theorem \ref{S4T2}; so $\overline{R}$ is a valuation domain.

Conversely, suppose that $\overline{R}$ is a valuation domain with maximal ideal $N$. Let $B$ be an overring of $R$. If $B$ is integral over $R$, then $B$ is an $n$-PVD by Theorem~\ref{S4T12}; so assume that $B$ is not integral over $R$. Let $b \in B\setminus \overline{R}$. Then $b^{-1} \in N$ since $\overline{R}$ is a valuation domain; so $m = b^{-n}  = (b ^{-1})^n \in M$ by Corollary~\ref{S4T13} since $\overline{R}$ is a valuation domain (and thus a PVD). Hence $1/m = b^n \in B$; so $B$ is an $n$-VD, and thus an $n$-PVD,  by Theorem~\ref{S4T8}.
The ``moreover'' statement is clear.
\end{proof}

 Let $R$ be a $1$-PVD (i.e., PVD) and $P$ a prime ideal of $R$. Then $A_1(P) = P$; so $V = (A_1(P) : A_1(P)) = (P : P)$ is a $1$-VD (i.e., valuation domain) by \cite[Proposition 4.3]{A2}, and it is easily checked that $P$ is the maximal ideal of $V$. We have the following analogous result for $n$-PVDs. 

 \begin{thm} \label{thm77}
 Let $R$ is an $n$-PVD, $P$  a prime ideal of $R$, and $I = (A_n(P))$. Then $V = (I : I)$ is an $n$-VD with maximal ideal $\sqrt{IV} = \{x \in V \mid x^n \in I\}$. Moreover, $\sqrt{IV} = \{ x \in V \mid x^n \in P \} = \sqrt{PV}$.
 \end{thm}

 \begin{proof}
 Let $x \in K$ with $x^n \not \in V$. Then $x^n \not \in P$; so $x^{-n}I \subseteq I$ by Corollary~\ref{S4C3}. Thus $x^{-n} \in V$; so $V$ is an $n$-VD with maximal ideal $N_V$. Let $y \in N_V$. Assume that $y^n \not \in I$; so $y^n \not \in P$. Thus $y^{-n}I \subseteq I$ by Corollary~\ref{S4C3} again; so $y^{-n} \in V$. Hence $y \in U(V)$, a contradiction. Thus $N_V \subseteq \{ x \in V \mid x^n \in I \} \subseteq \sqrt{IV}$. Also, $IV = I  \subsetneq V$; so $\sqrt{IV} \subseteq N_V$. Hence $N_V = \sqrt{IV} = \{x \in V \mid x^n \in I\}$. Clearly $\{ x \in V \mid x^n \in I \} \subseteq \{ x \in V \mid x^n \in P \}$ since $I \subseteq P$. Also, $x^n \in P$ for $x \in V \Rightarrow x^n \in A_n(P)$; so $\{ x \in V \mid x^n \in P \} \subseteq \{ x \in V \mid x^n \in I \}$. Thus $\{ x \in V \mid x^n \in I \} = \{ x \in V \mid x^n \in P \}$. Clearly $x \in P \Rightarrow x^n \in A_n(P) \subseteq I  \Rightarrow x^n \in \sqrt{IV}$; so $P \subseteq \sqrt{IV}$, and hence $\sqrt{PV} \subseteq \sqrt{IV}$. Also, $\sqrt{IV} \subseteq \sqrt{PV}$ since $I \subseteq P$; so $\sqrt{IV}  = \sqrt{PV}$.
 \end{proof}

Recall that a quasilocal integral domain $R$ with maximal ideal $M$ is a PVD if and only if $(M : M)$ is a valuation domain with maximal ideal $M$ \cite[Proposition 2.5]{AD}.  Example~\ref{ex33}(c) below shows that if $R$ is an $n$-PVD with maximal ideal $M$, then $(M : M)$ need not be an $n$-VD. And Example~\ref{ex33}(d)(e) shows that $V = (M : M)$ may be an $n$-VD with maximal ideal $\sqrt{MV} = \{ x \in V \mid x^n \in M \}$ when $R$ is not an $n$-PVD. However, since $M = A_1(M)$, the next theorem may be viewed as the $n$-PVD analog. By adding the extra condition ``$(*)_n$: if $x \in K$ is a nonunit of $\overline{R}$, then $x^n \in M$,''  we get a converse to Theorem ~\ref{thm77}. Note that $I = (A_n(M)) \subsetneq M$ in general (see Example~\ref{ex33}(a)(b)). 

\begin{thm} \label{S4T16}
Let $R$ be a quasilocal integral domain with maximal ideal $M$, quotient field $K$, and $I = (A_n(M))$. Then the following statements are equivalent.
\begin{enumerate}
\item $R$ is an $n$-PVD.
\item $V = (I : I)$ is an $n$-VD with maximal ideal $\sqrt{MV}  = \{x \in V \mid x^n \in M\}$, and if $x \in K$ is a nonunit of $\overline{R}$, then $x^n \in M$.
\end{enumerate}
\end{thm}

\begin{proof}
{\bf $(1) \Rightarrow (2)$}  By Theorem~\ref{thm77}, $V$ is an $n$-VD with maximal ideal $\sqrt{MV} = \{ x \in V \mid x^n \in M \}$. Let $x \in K$ be a nonunit of $\overline{R}$. Then $x ^n \in M$ by Corollary~\ref{S4T13}.

{\bf $(2) \Rightarrow (1)$}  Let $x \in K$. Suppose that  $x \in E_n(M)$, i.e., $x^n \not \in M$.  First, assume that  $x^n \in V$. Suppose that $x^n \in N = \{ x \in V \mid x^n \in M \}$; so $x^{n^2} = (x^n)^n \in M$. Thus $x \in \overline{R}$ and $x$ is a nonunit of $\overline{R}$; so $x^n \in M$ by hypothesis, a contradiction.  Hence $x^n \in U(V)$, and thus $x^{-n}I \subseteq I$. Hence $x^{-n}d \in I \subseteq M$ for every $d \in A_n(M)$. Now, suppose that $x^n \not\in V$. Then $x^{-n} \in V$ since $V$ is an $n$-VD. Thus $x^{-n}I \subseteq I$, and hence $x^{-n}d \in I \subseteq M$ for every $d \in A_n(M)$. Thus $x^{-n}d \in M$ for every $x \in E_n(M)$ and $d \in A_n(M)$; so $R$ is an $n$-PVD by Corollary~\ref{S4C4}.
\end{proof}

We end this section with several examples.

\begin{exa}  \label{ex33}
{\rm (a) Let $R = \mathbb{Z}_2[[X^2,X^3]] = \mathbb{Z}_2 + X^2\mathbb{Z}_2[[X]]$. Then $R$ is quasilocal with maximal ideal $M = (X^2,X^3) =X^2\mathbb{Z}_2[[X]]$ and quotient field $K = \mathbb{Z}_2[[X]][1/X]$. It is easily checked that $R$ is an $n$-PVD if and only if $n \geq 2$ and an $n$-VD if and only if $n$ is even. First, suppose that $n$ is even. Then  $I = (A_n(M)) = \mathbb{Z}_2X^n + X^{n+2}\mathbb{Z}_2[[X]] \subsetneq M$ and $V = (I : I) = R$ has maximal ideal $M_V = M$. 
Also, $M_V = \{ x \in V \mid x^n \in M \} \subsetneq   \{ x \in K \mid x^n \in M \} = X\mathbb{Z}_2[[X]]$. 
Next, suppose that $n \geq 3$ is odd. Then  $I = (A_n(M)) = X^n\mathbb{Z}_2[[X] \subsetneq M$ and $V = (I : I) = \mathbb{Z}_2[[X]]$ has maximal ideal $M_V = X\mathbb{Z}_2[[X]] = \{ x \in K \mid x^n \in M \}$.

(b)  Let $R = F[[X^2,X^3]] = F + X^2F[[X]]$, where $F$ is a field. Then $R$ is quasilocal with maximal ideal $M = (X^2,X^3) = X^2F[[X]]$ and quotient field $F[[X]][1/X]$, and $R$ is an $n$-PVD if and only if $n \geq 2$. If $char(F) = 2$, then $(A_n(M)) \subsetneq M$ for every integer $n \geq 2$. However, $M = (A_2(M))$ if $char(F) \neq 2$.

 (c) Let $R =  \mathbb{Z}_p + \mathbb{Z}_pX + X^2F[[X]]$, where $F = \overline{\mathbb{Z}_p}$ is the algebraic closure of $\mathbb{Z}_p$. Then $R$ is quasilocal with maximal ideal $M = \mathbb{Z}_pX + X^2F[[X]]$ and quotient field $K = F[[X]][1/X]$. Moreover, $R$ is an $n$-PVD if and only if $n \geq 2$ by Theorem ~\ref{S4T14} since $\overline{R} = F[[X]]$ is a PVD (in fact, a valuation domain). However, $V = (M : M) = \mathbb{Z}_p + XF[[X]]$ is an almost valuation domain with maximal ideal $XF[[X]] = \{ x \in K \mid x^n \in M \}$, but $V$ is not an $n$-VD for any positive integer $n$ by Example~\ref{ex22}(b). Note that $V$ is a PVD, and thus an $n$-PVD for every positive integer $n$

(d) Let $F$ be a field and $N$ a positive integer. Then $R_N = F + X^NF[[X]]$ is a quasilocal integral domain with maximal ideal $M_N = X^NF[[X]]$, quotient field $F[[X]][1/X]$, and integral closure $\overline{R_N} = F[[X]]$. Note that $V_N = (M_N : M_N) = F[[X]]$ is a valuation domain with maximal ideal $XF[[X]] = \{ x \in V_N \mid x^N \in M_N \} = \sqrt{M_NV_N}$, and thus $V_N$ is an $n$-VD for every positive integer $n$. However, $R_N$ is an $n$-PVD if and only if $n \geq N$, and $R_N$ satisfies condition $(*)_n$ if and only if $n \geq N$.

(e) Let $R = \mathbb{Z}_3 +\mathbb{Z}_3X^9 + X^{12}\mathbb{Z}_3[[X]]$. Then $R$ is a quasilocal integral domain with maximal ideal $M = \mathbb{Z}_3X^9 + X^{12}\mathbb{Z}_3[[X]]$, quotient field $\mathbb{Z}_3[[X]][1/X]$, and integral closure $\overline{R} = \mathbb{Z}_3[[X]]$. Note that $V = (M : M) = \mathbb{Z}_3 + X^3\mathbb{Z}_3[[X]]$ is a $3$-VD with maximal ideal $X^3\mathbb{Z}_3[[X]] = \sqrt{MV} = \{ x \in V \mid x^3 \in M \}$. However, $R$ is not a $3$-PVD since $(X^2)^3(X^2)^3 \in M$, but $X^6 \notin M$, and $R$ does not satisfy condition $(*)_3$ since $X^3 \notin M$.}
\end{exa}

\section{Psuedo $n$-strongly prime ideals, P$n$VDs, and $n$-VDs}

In this final section, we introduce and investigate pseudo $n$-valuation domains (P$n$VDs),  yet another generalization of PVDs. We also give some more results on $n$-VDs.

Let $R$ be an integral domain with quotient field $K$. Recall \cite{B2} that $R$  is a {\it  pseudo-almost valuation domain} (PAVD)  if every prime ideal $P$ of $R$ is {\it pseudo-strongly prime}, i.e., if whenever $xyP \subseteq P$ for $x, y \in K$, then there is a positive integer $n$ such that $x^n \in R$ or $y^nP \subseteq P$. Also, recall \cite{BH} that $R$ is an {\it almost pseudo-valuation domain} (APVD) if every prime ideal $P$ of $R$ is {\it strongly primary}, i.e, if whenever $xy \in P$ for $x, y \in K$, then $x^n \in P$ for some positive integer $n$ or  $y \in P$. Note that valuation domain $\Rightarrow$ PVD $\Rightarrow$APVD $\Rightarrow$ PAVD, and no implication is reversible \cite[page 1168]{B2}.

The following is an example of an $n$-PVD for some integer $n\geq 2$ which is neither an APVD, a PAVD, a PVD, nor an almost valuation domain.

\begin{exa}\label{S4E1}  {\rm (cf. \cite[Example 3.4]{B2})
Let $R = \mathbb{Q} + \mathbb{C}X^2 + X^4\mathbb{C}[[X]]$. Then $R$ is quasilocal with maximal ideal $M = \mathbb{C}X^2 + X^4\mathbb{C}[[X]]$ and quotient field $K = \mathbb{C}[[X]][1/X]$. One can see that $R$ is neither an APVD, a PAVD,  a PVD, an almost valuation domain, nor an $n$-VD for any positive integer $n$. However, it is easily checked that $R$ is a $n$-PVD for $n \geq 4$  and $\overline{R} = \overline{\mathbb{Q}} + X\mathbb{C}[[X]]$ is a PVD with maximal ideal $N = \{x \in K \mid  x^4 \in M\} = X\mathbb{C}[[X]]$, where $\overline{\mathbb{Q}}$ is the algebraic closure of $\mathbb{Q}$ in $\mathbb{C}$. Note that $\overline{R}$  is not a valuation domain; in fact, $\overline{R}$ is not an $n$-VD for any positive integer $n$, and $R$ is not an $n$-PVD for $n = 1, 2$, or  $3$.}
\end{exa}

We now give yet another ``$n$'' generalization of PVDs.

\begin{dfn} 
{\rm Let $R$ an integral domain with quotient field $K$. A prime ideal $P$ of $R$ is a {\it pseudo $n$-strongly prime ideal} of $R$ if whenever $xyP \subseteq P$ for  $x, y \in K$, then  $x^n \in R$ or $y^nP \subseteq P$. If every prime ideal  of $R$ is a pseudo $n$-strongly prime ideal of $R$, then $R$ is a {\it pseudo $n$-valuation domain} (P$n$VD).}
\end{dfn}

A P$1$VD is just a PVD \cite[Proposition 1.2]{HH1}, an $n$-VD is a P$n$VD, a P$n$VD is a PAVD, and a P$n$VD is also a P$(mn)$VD for every positive integer $m$. Moreover, from Theorem~\ref{S4T18} and Remark~\ref{R5.5}, it follows that a P$n$VD $R$ is quasilocal, the prime ideals of $R$ are linearly ordered by inclusion, and $dim(R) \leq 1$ when $R$ is Noetherian.  

The following is an example of a PAVD  which is not a P$n$VD  for any positive integer $n$.

\begin{exa} \label{newex}
{\rm Let $p$ be a positive prime integer and  $F = \overline{\mathbb{Z}_p}$ the algebraic closure of $\mathbb{Z}_p$. Then $R =  \mathbb{Z}_p +\mathbb{Z}_pX + X^2F[[X]]$ is quasilocal with maximal ideal $M = \mathbb{Z}_pX + X^2F[[X]]$ and quotient field $K = F[[X]][1/X]$. Let $y \in K$ with $y^n \not \in R$ for every positive integer $n$. Then $y = z/X^m$, where $z \in U(F[[X]])$ and $m \geq 0$. If $m > 0$, then $y^{-2}M \subseteq M$. If $m = 0$, then there is a positive integer $n$ such that $z(0)^n = 1$; so $y^{-n}M \subseteq M$. Thus $R$ is a PAVD by \cite[Lemma 2.1 and Theorem 2.5]{B2}. We now show that $R$ is not a P$n$VD for any positive integer $n$. For $n$ a positive integer,  there is a $b \in F$ with $b^n \notin \mathbb{Z}_p$ and $b^{-n} \notin \mathbb{Z}_p$. Hence $b^n \notin R$ and $b^{-n}X \notin M$; so $R$ is not a P$n$VD by Theorem~\ref{S4T18}(a)(b) below. However, $R$ is an $n$-PVD for every integer $n \geq 2$ by Example~\ref{ex33}(c).}
\end{exa}

The proofs of the following results are similar to the proofs given in \cite{B2}, and thus the details are left to the reader.

 \begin{thm} \label{S4T18}
 Let $R$ an integral domain with quotient field $K$. 
 
{\rm (a)} Let $P$ be a prime ideal of $R$. Then $P$ is a pseudo $n$-strongly prime ideal of $R$ if
and only if $x^{-n}P \subseteq P$ for every $x \in E_n(R)$ {\rm (}see \cite[Lemma 2.1]{B2}{\rm )}. 

{\rm (b)} $R$ is a P$n$VD if and only if $R$ is quasilocal with  pseudo $n$-strongly prime  maximal ideal {\rm (}see \cite[Theorem 2.8]{B2}{\rm )}. 

{\rm (c)}  $R$ is a P$n$VD if and only if for every $a, b \in R$, we have $a^n \, | \, b^n$ in $R$ or $b^n \, | \, a^nc$ in $R$ for every nonunit $c$ of $R$ {\rm (}see \cite[Proposition 2.9]{B2}{\rm)}. 

{\rm (d)}  Let $P$ be a prime ideal of $R$. If $R$ is a P$n$VD, then $R/P$ is a P$n$VD {\rm (}see \cite[Proposition 2.14]{B2}{\rm)}. 

{\rm (d)}  An $n$-root closed P$n$VD is a PVD {\rm (}see \cite[Theorem 2.13]{B2}{\rm)}. 
\end{thm} 

The next example gives some more $n$-PVDs that are not P$n$VDs.

\begin{exa}  \label{ES4.4}
{\rm Let $m \geq 2$ be an integer. Then $R = \mathbb{R} + \mathbb{R}X^{m-1} + X^m\mathbb{C}[[X]]$ is quasilocal with maximal ideal $M =  \mathbb{R}X^{m-1} + X^m\mathbb{C}[[X]]$, quotient field $K = \mathbb{C}[[X]][1/X]$, and integral closure $\overline{R} = \mathbb{C}[[X]]$. By Theorem~\ref{S4T14}, $R$ is an $n$-PVD for every integer $n \geq m$. For a positive integer $k$, let $y = e^{-i\pi/2k}$. Then $y^{k} = -i \not \in R$ and $y^{-k}X^{m-1} = iX^{m-1} \not \in R$; so $R$ is not a P$k$VD for any positive integer $k$ by Theorem~\ref{S4T18}(a).} 
\end{exa}

\begin{rmk} \label{R5.5}
{\rm Let $R$ an integral domain with quotient field $K$. Since $A_n(P) \subseteq P$ for every prime ideal $P$ of $R$,  every pseudo $n$-strongly prime ideal of $R$ is also an $n$-powerful semiprimary ideal of $R$ by Corollary~\ref{S4C3} and Theorem~\ref{S4T18}(a), and thus a P$n$VD is an $n$-PVD. Hence, we have the following implications
 $$n{\rm -}VD \Rightarrow PnVD \Rightarrow n{\rm -}PVD.$$
Neither of the above two implications is reversible. A P$n$VD need not be an $n$-VD by Theorem~\ref{S4T21}, and an $n$-PVD need not be a P$n$VD by Examples~\ref{newex} and \ref{ES4.4}. Also, note that the ring in Example~\ref{S4E1} is a $4$-PVD, but not a P$4$VD}.
\end{rmk}

The next theorem gives a case where an $n$-PVD is a P$n$VD. Note that the $n = 1$ case is just \cite[Proposition 2.5]{AD}. We may have $M \neq (A_n(M))$ for every integer $n \geq 2$ (see Example~\ref{ex33}(a)(b)). Note that in the next two theorems, we need the extra condition $(*)_n$ (cf. Example~\ref{ex33}(d)(e), and recall that if $R$ is not an $n$-PVD, then $R$ is not a P$n$VD by Remark~\ref{R5.5}). 

\begin{thm}\label{S4T19}
Let $R$ be a quasilocal integral domain with maximal ideal $M = (A_n(M))$ and quotient field $K$. Then the following statements are equivalent.
\begin{enumerate}
\item $R$ is a P$n$VD.
\item $R$ is an $n$-PVD.
\item $V = (M : M)$ is an $n$-VD with maximal ideal $\sqrt{MV} = \{x \in V \mid x^n \in M\}$, and if $x \in K$ is a nonunit of $\overline{R}$, then $x^n \in M$.
\end{enumerate}
\end{thm}

\begin{proof}
$(1) \Rightarrow (2)$ A P$n$VD is an $n$-PVD by Remark~\ref{R5.5}.

 $(2) \Rightarrow (1)$  Let $x \in E_n(R)$; so $x \in E_n(M)$. Then $x^{-n}A_n(M) \subseteq M$ by Corollary~\ref{S4C4}, and thus $x^{-n}M \subseteq M$ since $M = (A_n(M))$ by  hypothesis. Hence $R$ is a P$n$VD by Theorem~\ref{S4T18}(a)(b).

$(3) \Leftrightarrow (1)$ This is clear by Theorem \ref{S4T16}.
\end{proof}

The following result recovers that a quasilocal integral domain $R$ with maximal ideal $M$ is a PVD if and only if $(M : M)$ is a valuation domain with maximal ideal $M$ \cite[Proposition 2.5]{AD}; its proof is an analog of the proof of \cite[Theorem 2.15]{B2}.

\begin{thm} \label{S4T29}
Let $R$ be a quasilocal integral domain with maximal ideal $M$ and quotient field $K$. Then the following statements are equivalent.
\begin{enumerate}
\item $R$ is a P$n$VD.
\item $V = (M : M)$ is an $n$-VD with maximal ideal $\sqrt{MV} = \{x \in V \mid x^n \in M\}$, and if $x \in K$ is a nonunit of $\overline{R}$, then $x^n \in M$.
\end{enumerate}
\end{thm}

\begin{proof}
 {$(1) \Rightarrow (2)$} Let $R$ be a P$n$VD. Let $x \in E_n(V)$; so $x \in E_n(R)$. Then $x^{-n}M \subseteq M$ by Theorem \ref{S4T18}(a); so $x^{-n} \in V$. Thus $V$ is an $n$-VD with maximal ideal $M_V$. Let $x \in M_V$. If  $x^n  \in R$, then $x^n \in M$. Otherwise, $x \in E_n(R)$. Hence $x^{-n}M \subseteq M$ by Theorem \ref{S4T18}(a) again; so      $x^{-n}  \in V$. Thus $x \in U(V)$, a contradiction. Hence $M_V \subseteq \{ x\in V \mid x^n \in M \} \subseteq \sqrt{MV}$, and $\sqrt{MV} \subseteq M_V$ since $MV = M \subsetneq V$. Thus $ M_V = \sqrt{MV} = \{x \in V \mid x^n \in M\}$. If $x \in K$ is a nonunit of $\overline{R}$, then $x^n \in M$ by Corollary~\ref{S4T13} since a P$n$VD is an $n$-PVD by Remark~\ref{R5.5}.

{$(2) \Rightarrow (1)$} Let $V = (M : M)$ be an $n$-VD with maximal ideal $\sqrt{MV} = \{x \in V \mid x^n \in M\}$.
Suppose that $x \in E_n(R)$; so $x^n \notin M$. If $x^n \in  V$ and $x^n \notin U(V)$, then $x^{n^2} = (x^n)^n \in M \subseteq R$; so $x \in \overline{R}$. Thus $x^n \in M$ by hypothesis, a contradiction. Hence $x^n \in U(V)$; so $x^{-n}M \subseteq M$. If  $x^n \notin V$, then
 $x^{-n} \in V$ since $V$ is an $n$-VD. Thus $x^{-n}M \subseteq M$ in either case; so $R$ is a P$n$VD by Theorem~\ref{S4T18}(a)(b).
\end{proof}

\begin{cor} \label{cor66}
Let $R$ be a P$n$VD with maximal ideal $M$. If $P$ is a prime ideal of $R$, then $W_P = (P : P)$ is an $n$-VD. Moreover, if $P \subseteq Q$ are prime ideals of $R$, then $W_Q = (Q : Q) \subseteq (P : P) = W_P$.
\end{cor}

\begin{proof}
We have $V = (M : M) \subseteq (P : P) = W_P$ by \cite[Lemma 2.2]{A} since $P \subseteq M$. Thus $W_P$ is an $n$-VD  since $V$ is an $n$-VD by Theorem~\ref{S4T29}.  The ``moreover'' statement 
is clear since $(Q : Q) \subseteq (P : P)$ by \cite[Lemma 2.2]{A} again.
\end{proof}

Let $T$ be an overring of an integral domain $R$ and $n$ a positive integer. Then $T$ is an {\it $n$-root extension}  of $R$ if $x^n \in R$ for every $x \in T$, and $T$ is a  {\it root extension} of $R$ if for every $x \in T$, there is a positive integer $m$ such that $x^m \in R$.

\begin{thm}\label{S4T17}
Let $R$ be a quasilocal integral domain with maximal ideal $M$ and quotient field $K$, $n$ a positive integer, and $V$  a valuation overring of $R$ with maximal ideal $N = \{x \in V \mid x^n \in M\}$. Then $R$ is an $n$-VD if and only if $V$ is an $n$-root extension of $R$.
\end{thm}

\begin{proof}
We may assume that $R \subsetneq V$. 
Suppose that $R$ is an $n$-VD. Let $x \in V\setminus R$. If $x \in N$, then
$x^n \in  M \subseteq R$. Thus, assume that $x \not \in N$. Since $N$ is the maximal ideal of $V$, we have $x \in U(V)$. Thus $x^n \not\in M$ and $x^{-n} \
\not\in M$. Since $R$ is an $n$-VD, we have $x^n \in U(R) \subseteq R$. Hence $V$ is an $n$-root extension of $R$.

Conversely, suppose that $V$ is an $n$-root extension of $R$. Let  $x \in K$ with $x^n \not\in R$. Then $x \not \in V$ since $V$ is an $n$-root extension of $R$, and thus $x^{-1} \in V$ since $V$ is a valuation domain.  Hence $x^{-n} \in R$ since $V$ is an $n$-root extension of $R$, and thus $R$ is an $n$-VD. 
\end{proof}

\begin{lem}\label{S4L1}
Let $R$ be a quasilocal integral domain with maximal ideal $M$ and quotient field $K$. If $R$ is an $n$-VD, then $\overline{R}$ is a valuation domain with maximal ideal $\sqrt{M\overline{R}} = \{x \in K \mid x^n \in M\}$ and $R\subseteq \overline{R}$ is an $n$-root extension.
\end{lem}

\begin{proof}
Let $R$ be an $n$-VD. Then $R$ is an almost valuation domain; so $\overline{R}$ is a valuation domain and $R\subseteq \overline{R}$ is a root extension by \cite[Theorem 5.6]{AZ1}. Thus $\sqrt{M\overline{R}}  = \{x \in K \mid x^n \in M\}$ is the maximal ideal of $\overline{R}$ by Theorem \ref{S4T14} since an $n$-VD is an $n$-PVD. Hence $\overline{R}$ is an $n$-root extension of $R$ by Theorem \ref{S4T17}.
\end{proof}

\begin{thm}\label{S4T20}
Let $R$ be a quasilocal integral domain with maximal ideal $M$ and quotient field $K$, and let $V$ be an $n$-VD  overring of $R$ with maximal ideal $N = \{x \in V \mid x^n \in M\}$. Then $R$ is an $n$-VD if and only if \, $\overline{V} = \overline{R} = \{x \in K \mid x^n \in R\}$. 
\end{thm}

\begin{proof} 
 We may assume that  $R \subsetneq V$.
Suppose that $R$ is an $n$-VD. Then $\overline{R}$ is a valuation domain with maximal ideal $W = \{x \in K \mid x^n \in M\}$ and $R\subseteq \overline{R}$ is an $n$-root extension by Lemma \ref{S4L1}. Similarly, since $V$ is an $n$-VD, $\overline{V}$ is a valuation domain with maximal ideal $T = \{x \in K \mid x^n \in N\}$ and $V\subseteq \overline{V}$ is an $n$-root extension by Lemma \ref{S4L1}. First, we show that $R\subsetneq V$ is an $n$-root extension. Let $x \in V\setminus R$. If $x \in N$, then 
$x^n \in  M \subseteq R$.  Hence, assume that $x \not \in N$. Since $N$ is the maximal ideal of $V$, we have $x \in U(V)$. Since $x \in U(V)$, neither $x^n \in M$ nor $x^{-n} \in M$. Since $R$ is an $n$-VD, $x^n \in U(R) \subseteq R$. Thus $V$ is an $n$-root extension of $R$. Since $V$ is an integral overring of $R$, we have that $\overline{V}$ is integral over $R$, and thus  $\overline{R} = \overline{V} = \{x \in K \mid x^n \in R\}$.

Conversely, suppose that $\overline{R} = \overline{V} = \{x \in K \mid x^n \in R\}$, and let $x \in K$ with $x^n \not\in R$. Then $x \not \in \overline{V}$, and thus $x^{-1} \in \overline{V}$ since $\overline{V}$ is a valuation domain by Lemma~\ref{S4L1}. Hence $x^{-n} \in R$; so $R$ is an $n$-VD.
\end{proof}

Let $V$ be a valuation domain with maximal ideal $M$, residue field $F = V/M$, and $\pi : V \longrightarrow F$ the canonical epimorphism. If $k$ is a subfield of $F$, then $R = \pi^{-1}(k)$ is a PVD with maximal ideal $M$  \cite[Proposition 2.6]{AD}. Moreover, every PVD arises in this way. Let $R$ be a PVD with maximal ideal $M$. Then $V = (M:M)$ is a valuation domain with maximal ideal $M$ \cite[Proposition 2.5]{AD}; so $R= \pi^{-1}(R/M)$. A similar result holds for P$n$VDs and $n$-VDs.

 \begin{thm}\label{S4T21}
 Let $V$ be an $n$-VD with nonzero maximal ideal
$M$,  residue field $F = V/M$, $\pi : V \longrightarrow F$ the canonical epimorphism, $k$  a subfield of $F$, and $R = \pi^{-1}(k)$. Then the pullback
$R = V \times_F k$ is a P$n$VD with maximal ideal $M$. In particular, if $k$ is properly contained
in $F$ and $V$ is not an $n$-root extension of $R$, then $R$ is a P$n$VD which is not an $n$-VD. 
\end{thm}

\begin{proof} 
In view of the construction stated in the hypothesis, it is well known that $M$ is a maximal ideal of $R$
for any integral domain $V$. Also, it is clear that $R$ and $V$ have
the same quotient field $K$ by \cite[Lemma 3.1]{AD}. Let $x \in E_n(R)$. Then  $x^n \in V$ or $x^{-n} \in V$ since $V$ is an $n$-VD. Suppose that $x^n \in V$. Since $x \in E_n(R)$ and $M$ is the maximal ideal of $R$, we have $x^n \not \in M$. Thus $x^n \in U(V)$, and hence $x^{-n} \in V$; so $x^{-n}M \subseteq M$ since $M$ is an ideal of $V$. Now suppose that $x^{-n} \in V$. 
Then $x^{-n}M\subseteq M$ since $M$ is an ideal of $V$. Thus $M$ is a pseudo $n$-strongly prime ideal of $R$ by Theorem~\ref{S4T18}(a), and hence $R$ is a P$n$VD by Theorem~\ref{S4T18}(b). The remaining part is clear from Theorem~\ref{S4T20}.
\end{proof}

The final example illustrates the previous theorem.

\begin{exa}
{\rm (a)  Let $V = \mathbb{Z}_p(t)[[X]]$. Then $V$ is a valuation domain; so $R = \mathbb{Z}_p + X\mathbb{Z}_p(t)[[X]]$ is a P$n$VD for every positive integer $n$, but not an $n$-VD for any positive integer $n$, by Theorem~\ref{S4T21} since $V$ is not an $n$-root extension of $R$. Note that $R$ is actually a PVD.

(b) Let $T = K + M$ be a quasilocal integral domain with maximal ideal $M$ and $K$ a subfield of $T$. Let $k$ be a subfield of $K$ and $R = k + M$. Then $R$ is also quasilocal with maximal ideal $M$. Thus $R$ is an $n$-PVD (resp., P$n$VD) if and only if $T$ is an $n$-PVD (resp., P$n$VD) by Corollary~\ref{S4C1} (resp., Theorem~\ref{S4T18}(b)).

For example, $T = \mathbb{R}[[X^2,X^3]] = \mathbb{R} + X^2\mathbb{R}[[X]]$ is an $n$-PVD $\Leftrightarrow$ $n \geq 2$ (Example~\ref{ex33}(b)), and thus $R =\mathbb{Q} + X^2\mathbb{R}[[X]]$ is an $n$-PVD $\Leftrightarrow$ $n \geq 2$.}
\end{exa}


\begin{thebibliography}{999}

\bibitem{AA} D. D. Anderson and D. F. Anderson, Multiplicatively closed subsets of fields,  Houston J. Math. 13(1987), 1--11.

\bibitem{A?} D. D. Anderson and M. Winders, Idealization of a module,  J. Commut. Algebra 1(2009), 3--56.

\bibitem{AKL} D. D. Anderson, K. R. Knopp. and R. L. Lewin, Almost Bezout domains, II,  J. Algebra 167(1994), 547--556.

\bibitem{AZ1} D. D. Anderson and M. Zafrullah, Almost Bezout domains,  J. Algebra 142(1991), 285--309.

\bibitem{AZ2} D. D. Anderson and M. Zafrullah, Almost Bezout domains, III, Bull. Math. Soc. Sci. Math. Roumanie 51(2008), 3--9.

\bibitem{A} D. F. Anderson, Comparability of ideals and valuation overrings, Houston J. Math. 5(1979), 451--463.

\bibitem{A3} D. F. Anderson, Root closure in integral domains,  J. Algebra 79(1982), 51--59.

\bibitem{A2} D. F. Anderson, When the dual of an ideal is a ring, Houston J. Math. 9(1983), 325--332.

 \bibitem{AB1} D. F. Anderson and A. Badawi, On $n$-absorbing ideals of commutative rings, Comm. Algebra 39(2011), 1646--1672.

\bibitem{AB2} D. F. Anderson and A. Badawi, On (m,n)-closed ideals of commutative rings, J. Algebra Appl. 16(2017)(21 pages).

\bibitem{AD} D. F. Anderson and D. E. Dobbs, Pairs of rings with the same prime ideals, Canad. J. Math. 32(1980), 362--384.

\bibitem{AD2} D. F. Anderson and D. E. Dobbs, On the product of ideals, Canad. Math. Bull. 26(1983), 238--246.

\bibitem{B} A. Badawi, On abelian $\pi$-regular rings, Comm. Algebra 25(1997), 1009--1021.

\bibitem{B?} A. Badawi, On divided commutative rings, Comm. Algebra 27(1999), 1465--1474.

\bibitem{B1} A. Badawi, On 2-absorbing ideals of commutative rings, Bull. Austral. Math. Soc. 75 (2007), 417--429.

\bibitem{B2} A. Badawi, On pseudo-almost valuation domains, Comm. Algebra 35(2007), 1167--1181.

\bibitem{BH} A. Badawi and E. Houston, Powerful ideals, strongly primary ideals, pseudo-valuation domains, and conducive domains, Comm. Algebra 30(2002), 1591--1606.

\bibitem {CW} S. H. Choi and A. Walker, The radical of an $n$-absorbing ideal, arXiv:1610.10077 [math.AC].

\bibitem{CH} J. A. Cox and A. J. Hetzel, Uniformly primary ideals, J. Pure. Appl. Algebra 212(2008), 1--8.

\bibitem{Do} D. E. Dobbs, Divided rings and going-down, Pacific J. Math 67(1976), 353--363.

\bibitem{Do2} D. E. Dobbs, Coherence, ascent of going-down, and pseudo-valuation domains, Houston J. Math. 4(1978), 551--576.

\bibitem{D} G. Donadze, A proof of the Anderson–Badawi $rad(I)^n \subseteq I$ formula for n-absorbing ideals, Proc. Indian Acad. Sci. (Math. Sci.) 128(2018), 6--11.

\bibitem{G} R. Gilmer, {\it Multiplicative Ideal Theory}, Marcel Dekker, Inc., New York, 1972.

\bibitem{HH1} J. R. Hedstrom and E. G. Houston, Pseudo-valuation domains, Pacific J. Math. 75(1978), 137--147.

\bibitem{HH2} J. R. Hedstrom and E. G. Houston, Pseudo-valuation domains, II, Houston J. Math. 4(1978), 199--207.

\bibitem{H} J. Huckaba, {\it Rings with Zero Divisors}, Marcel Dekker, Inc., New York, 1988.

\bibitem{K} I. Kaplansky, {\it Commutative Rings, Revised Edition}, The University of Chicago Press, Chicago and London, 1974.

\bibitem{LM} M. D. Larson and P. J. McCarthy, {\it Multiplicative Theory of Ideals}, Academic Press, New York/London, 1971.

\bibitem{MMM} N. Mahdou, A. Mimouni, and M. Moutui, On almost valuation and almost Bezout rings, Comm. Algebra 43(2015), 297--308.
\end{thebibliography}
\end{document}